\documentclass[12pt,reqno]{amsart}
\usepackage[utf8]{inputenc}

\usepackage{amssymb,latexsym, mathtools,tikz}
\usepackage{enumerate}
\usepackage{graphicx}
\usepackage[mathscr]{euscript}
\usepackage{float}
\usepackage{placeins}
\usepackage{mdframed}
\usepackage{amssymb}
\usepackage{esint}
\usepackage{cool}
\usepackage[all,cmtip]{xy}
\usepackage{mathtools}
\usepackage{amstext} % for \text macro
\usepackage{array}   % for \newcolumntype macro
\usepackage[shortlabels]{enumitem}
\usepackage{ytableau}
\usepackage{tikz}
\usetikzlibrary{arrows,decorations.markings, cd}
\tikzstyle arrowstyle=[scale=1]
\tikzstyle directed=[postaction={decorate,
decoration={markings,mark=at position .65 with {\arrow[arrowstyle]{stealth}}}}]
\usepackage{mathdots}
\usepackage{quiver}

\usepackage[bookmarks, bookmarksdepth=2, colorlinks=true, linkcolor=blue, citecolor=blue, urlcolor=blue]{hyperref}

\usepackage{rotating}
\usepackage{calligra,mathrsfs}
\usepackage{comment}

\usepackage[nobysame, alphabetic, citation-order]{amsrefs} 
\AtBeginDocument{%
   \def\MR#1{}
 }

\usepackage[margin=1in]{geometry}

\newcolumntype{L}{>{$}l<{$}} % math-mode version of "l" column type
\newcolumntype{C}{>{$}c<{$}}
\numberwithin{equation}{section}
\newtheorem{theorem}[equation]{Theorem}
\newtheorem{lemma}[equation]{Lemma}

\newtheorem{cor}[equation]{Corollary}
\newtheorem{prop}[equation]{Proposition}

\theoremstyle{definition}

\newtheorem{rmk}[equation]{Remark}
\newenvironment{remark}[1][]{\begin{rmk}[#1] \pushQED{\qed}}{\popQED \end{rmk}}
\newtheorem{eg}[equation]{Example}
\newenvironment{example}[1][]{\begin{eg}[#1] \pushQED{\qed}}{\popQED \end{eg}}
\newtheorem{defn}[equation]{Definition}
\newenvironment{definition}[1][]{\begin{defn}[#1]\pushQED{\qed}}{\popQED \end{defn}}
\newtheorem{notn}[equation]{Notation}
\newenvironment{notation}[1][]{\begin{notn}[#1]\pushQED{\qed}}{\popQED \end{notn}}
\newtheorem{cnst}[equation]{Construction}
\newenvironment{construction}[1][]{\begin{cnst}[#1]\pushQED{\qed}}{\popQED \end{cnst}}

\newcommand{\R}{\mathcal{R}}
\newcommand{\cat}[1]{\mathcal{#1}}

\newcommand{\defi}[1]{{\bf\upshape\sffamily #1}}
\newcommand{\col}{\operatorname{col}}

\newcommand{\ext}{\operatorname{Ext}}

\newcommand{\tor}{\operatorname{Tor}}

\newcommand{\Ker}{\operatorname{Ker}}

\newcommand{\ideal}[1]{\mathfrak{#1}}

\newcommand{\fn}{\ideal{n}}
\newcommand{\fp}{\ideal{p}}

\newcommand{\fb}{\ideal{b}}

\renewcommand{\phi}{\varphi}

\newcommand{\bbs}{\mathbb{S}}
\newcommand{\bbz}{\mathbb{Z}}

\newcommand{\bbq}{\mathbb{Q}}

\newcommand{\bbc}{\mathbb{C}}

\renewcommand{\sl}{\mathfrak{sl}}

\newcommand{\Spo}{\operatorname{SpO}}

\renewcommand{\geq}{\geqslant}
\renewcommand{\leq}{\leqslant}
\renewcommand{\ker}{\Ker}
\renewcommand{\hom}{\Hom}

\newcommand{\Hom}{\operatorname{Hom}}	
\newcommand{\Tor}[4][R]{\operatorname{Tor}^{#1}_{#2}(#3,#4)}

\newcommand{\rad}[1]{\operatorname{rad}(#1)}

\newcommand{\SO}{\operatorname{SO}}

\newcommand{\GL}{\operatorname{{GL}}}
\newcommand{\gl}{\mathfrak{gl}}
\newcommand{\so}{\mathfrak{so}}

\def\Tor{\operatorname{Tor}}

\newcommand{\maps}[5]{\xymatrix{#1 \ar[r]^-{#3} & #2 \\
#4 \ar@{|->}[r] & #5 \\}}

\newcommand{\ind}{\textrm{ind}}

\newcommand*{\sheafhom}{\mathscr{H}\text{\kern -3pt {\calligra\large om}}\,}
\newcommand*{\sheafext}{\mathscr{E}\text{\kern -3pt {\calligra\large xt}}\,}
\newcommand{\fh}{\mathfrak{h}}

\newcommand{\hs}{\operatorname{HS}}
\newcommand{\uU}{\mathrm{U}}

\def\Tor{\operatorname{Tor}}

\def\kk{\mathbf{k}}

\def\w{\wedge}

\newcommand{\Sym}{\operatorname{Sym}}

\newcommand{\igr}{\operatorname{IGr}}
\newcommand{\g}{\mathfrak{g}}

\newcommand{\OO}{\cat{O}}

\newcommand{\row}{\operatorname{row}}
\renewcommand{\col}{\operatorname{col}}
\renewcommand{\ind}{\operatorname{Ind}}

\newcommand{\BGG}{\mathrm{BGG}}

\title{From total positivity to pure free resolutions}
\author{Steven V Sam}
\address{Department of Mathematics, University of California San Diego}
\email{ssam@ucsd.edu}

\author{Keller VandeBogert}
\address{Department of Mathematics, University of Notre Dame, Notre Dame}
\email{kvandebo@nd.edu}

\date{March 25, 2025}

\begin{document}

\begin{abstract}
  Using the Jacobi--Trudi identity as a base, we establish parallels between the theory of totally positive integer sequences and Koszul algebras. We then focus on the case of quadric hypersurface rings and use this parallel to construct new analogues of Schur modules. We investigate some of their Lie-theoretic properties (and in more detail in a followup article) and use them to construct pure free resolutions for quadric hypersurface rings which are completely analogous to the construction given by Eisenbud, Fl\o ystad, and Weyman in the case of polynomial rings.
\end{abstract}

\maketitle

\section{Introduction}

The motivations behind this article come from commutative algebra (free resolutions), representation theory (quantum algebras), and combinatorics (total positivity). At the center is a single construction that ties all of these topics together. While we feel we are still far from understanding its full ramifications, what we have discovered suggests that there is a very rich mathematical theory to be uncovered. We will begin with the combinatorial discussion. 

\subsection{A global criterion for equivariant positivity}

A real-valued matrix is totally positive if all of its minors are non-negative, and a sequence of real numbers $(a_d)_{d \ge 0}$ is a P\'olya frequency (PF) sequence if the Toeplitz matrix $(a_{i-j})$ (with the convention that $a_d=0$ for $d<0$) is totally positive. The universal setting for this is when the $a_d$ are algebraically independent variables; concretely, we can take $a_d$ to be the complete homogeneous symmetric function of degree $d$. In this case, each minor is a skew Schur function by the classical Jacobi--Trudi identity.

The complete homogeneous symmetric functions are well-known to be the characters of symmetric powers, while the skew Schur functions are characters of skew Schur functors. The sum of the symmetric powers is the symmetric algebra (or polynomial ring), and the starting point for us is to think of this as a recipe for associating a theory of ``Schur modules'' to a graded algebra. At the most basic level, if $(a_d)$ is a PF sequence for the Hilbert function of an algebra $A$, then the minors should be the dimensions of the corresponding Schur modules.

The second author explored this idea in \cite{vandebogert2023ribbon} where it is proven that if $A$ is a Koszul algebra, then there is a good functorial construction of Schur modules indexed by ribbon-shaped skew diagrams. Unfortunately, this idea cannot extend more generally since the Hilbert function of a Koszul algebra need not be a PF sequence. Surprisingly, the converse is true, which is our first main theorem (see Theorem~\ref{thm:GPFalgebraStructure}).

\begin{theorem}
  If an integer-valued sequence $(a_d)_{d \ge 0}$ with $a_0=1$ is a PF sequence, then there exists a Koszul algebra $A$ with $\dim A_d = a_d$.
\end{theorem}

In recent years, there has been considerable interest in \emph{equivariant} notions of total positivity; that is, positivity at the level of K-theory and not just as integer sequences \cites{lam2007schur,gedeon2017equivariant,matherne2023equivariant,gui2024equivariant,angarone2025chow}. This finer notion of positivity is much rarer and it is also more difficult to prove since there is no equivariant analogue of the famous AESW theorem \cite{aissen1952generating}. One of the main contributions of this work is a global criterion for checking equivariant total positivity, which comes from the existence of a so-called \defi{Jacobi--Trudi structure}. Namely, we define a Jacobi--Trudi structure on a $G$-representation $A$ to be a series of compatible $\GL_n$-actions on $A^{\otimes n}$ that commute with $G$ and such that each $\mu$-weight space coincides with the tensor product $A_{\mu_1} \otimes \cdots \otimes A_{\mu_n}$. Then for each partition $\lambda$, the dimension of the $\bbs_\lambda(\bbc^n)$-isotypic component of $A^{\otimes n}$ is precisely the dimension given by the minor above. Hence we get a functor $\Phi_{A,n} = (- \otimes A^{\otimes n})^{\GL_n}$ on $\GL_n$-representations inducing an equivalence between the notion of equivariant positivity and the existence of Jacobi--Trudi structures:

\begin{theorem}
    Let $A = \bigoplus_{i \geq 0} A_i$ be a graded $G$-representation, where $A_0 = \kk$ is the trivial representation. Then:
    $$A \ \text{is equivariantly totally positive} \quad \iff \quad A \ \text{admits a Jacobi--Trudi structure.}$$
\end{theorem}

The key ingredient underlying all of the above results is a construction of Zelevinsky \cite{zelevinskii1987resolvents}, which has proven to be a powerful tool for endowing purely combinatorial data with algebraic structure. Moreover, there is a direct parallel between partial Jacobi--Trudi structures and the partial $\operatorname{PF}_\ell$-properties \cite{reiner2005charney} (for instance, equivariant \defi{log-concavity} is equivalent to a compatible $\GL_2$-action on $A^{\otimes 2}$), a direction that will be studied much more closely in future work of the authors. This suggests that there are natural hierarchies of Koszul algebras that should be studied in parallel to the positivity properties of integer sequences. We also refer to \cite{reiner2005charney} where similar types of questions are asked.

Next, we will outline the surprising significance of this construction in the context of Boij--S\"oderberg theory.

\subsection{Pure free resolutions}

In \cite{boij2008graded}, Boij and S\"oderberg conjectured that the  Betti tables of graded Cohen--Macaulay modules over polynomial rings $A$ can be written as non-negative linear combinations of the Betti tables of modules $M$ having a pure free resolution. This latter condition means that each Tor group $\Tor_i^A(M,\kk)$, which is naturally a graded vector space, is concentrated in a single degree $d_i$. It was previously known that the sequence of these degrees, which we will call the degree sequence, uniquely determines the ranks of the Tor groups up to simultaneous scalar multiple. However, the existence of modules having a pure free resolution with a given degree sequence was not known; the first construction (EFW complexes) was given in \cite{eisenbud2011existence} using ideas from representation theory. The full statement of the conjecture was later resolved in \cite{eisenbud2009betti}.

A natural generalization is to replace the polynomial ring with another graded commutative ring $A$. Our belief is that the next most complicated class from this perspective are the rings with finite Cohen--Macaulay representation type, i.e., those rings for which the set of isomorphism classes of indecomposable maximal Cohen--Macaulay modules is finite, since the ``tail end'' of every free resolution of a graded $A$-module is precisely the resolution of a maximal Cohen--Macaulay module, at least if $A$ is itself Cohen--Macaulay. In particular, the behavior of the infinite aspect of minimal free resolutions is reduced to a discrete set of cases. Such rings have been classified \cite{eisenbud1988classification} (see also Example~\ref{ex:ver-schurdim} for a curious connection of this classification and Jacobi--Trudi structures). The two infinite families are quadric hypersurface rings and coordinate rings of the rational normal curve, along with a finite set of exceptional examples. The case of the rational normal curves was worked out in \cite{BS-RNC} and one of the exceptional examples was worked out in \cite{BS-3pts}, so the final unsolved infinite family is the case of quadric hypersurface rings. 

It turns out that the functors $\Phi_{A,n}$ induced by a general Jacobi--Trudi structure behave in rather unpredictable ways, which means that a more careful analysis of these functors in the quadric case is required to construct pure free resolutions. We thus take advantage of the explicit Jacobi--Trudi structure constructed in \cite{charCoincidence}, combined with Lie-theoretic analogues of the geometric technique for computing syzygies in \cite{weyman2003}, to deduce the following:

\begin{theorem}\label{thm:introPhifunctors}
    Let $\kk$ be a field of characteristic $0$ and let $A = \kk[x_1,\dots,x_m]/(q)$ where $q$ is a nonzero homogeneous quadratic polynomial of odd rank. Then the functors $\Phi_{A,n}$ induced by the Jacobi--Trudi structure on $A$ (provided by \cite{charCoincidence}) convert $\kk[x_1,\dots,x_n]$-modules into $A$-modules and preserve freeness (if $n$ is large enough with respect to the representation-type of the generators).
  \end{theorem}
  
There is a subtle point worth mentioning now: the functors $\Phi_{A,n}$ \emph{cannot} be monoidal in general (see Remark~\ref{rk:noMonoidalStructure}), so the above result is in some sense the best possible. Combining Theorem \ref{thm:introPhifunctors} with the construction of EFW complexes, we construct pure free resolutions over odd rank quadric hypersurface rings for all degree sequences (see Theorem~\ref{thm:quadricPieriResolutions} for the detailed version).

\begin{theorem}
  Let $\kk$ be a field of characteristic $0$ and let $A = \kk[x_1,\dots,x_m]/(q)$ where $q$ is a nonzero homogeneous quadratic polynomial of odd rank.
  Pick a sequence of positive integers $e_0, e_1 , \dots , e_{m-1}$ and set $d_i = \sum_{j=0}^i e_j$.
  \begin{enumerate}
  \item There is a finite length pure free resolution over $A$ with degree sequence

    $d_0,d_1 , d_2 , \dots , d_{m-1}$.
  \item There is an infinite length pure free resolution over $A$  with degree sequence

    $d_0,d_1 , \dots , d_{m-1} , d_{m-1}+1 ,d_{m-1}+2 , \dots$.
  \end{enumerate}
  In both cases, the module being resolved has finite length.
\end{theorem}
            
We also show that the pure free resolutions for the rational normal curve (as studied in \cite{BS-RNC}) can be understood in terms of Jacobi--Trudi structures. 

Aside from its direct connection to the Boij--S\"oderberg conjectures, another value of this result is that it is a novel new way to construct explicit infinite minimal free resolutions. For instance, the well-known Eisenbud--Shamash construction \cite{eisenbud1980homological,shamash1969poincareseries} gives free resolutions over hypersurface rings starting from free resolutions over polynomial rings, but the result is often non-minimal.

The representation theory portion of our discussion expanded the scope of the article quite a bit and is addressed in a separate article \cite{charCoincidence}.

\subsection{Future Directions}

We discuss some future directions related to constructing ``well-behaved'' Jacobi--Trudi structures. In general, there are many different classes of algebras whose Hilbert functions are PF sequences. The results of this paper give a canonical method of endowing any totally positive algebra with a Jacobi--Trudi structure, but our description in terms of multiplicity spaces does not take into account the ambient product structure on $A$; this leads to a generally badly-behaved functorial correspondence $\bbs_\lambda (E) \mapsto \bbs_\lambda^A$. Our construction of the Jacobi--Trudi structure on quadric hypersurfaces hints at a deeper relationship with the \defi{homotopy Lie algebra} of $A$. 

Recall that for any commutative ring $A$, the Ext algebra $\ext^\bullet_A (\kk , \kk)$ is the enveloping algebra $\uU (\pi^\bullet (A))$ of a $\bbz$-graded Lie superalgebra $\pi^\bullet (A)$ \cite{milnor1965structure}, which is the homotopy Lie algebra of $A$. Given any vector space $E$, consider the ``fattened'' homotopy Lie algebra
\[
  \pi_E^\bullet (A) \coloneq \bigoplus_{d \ge 1} \pi^d (A) \otimes \bigwedge^d E 
\]
with bracket $[ u \otimes e , u' \otimes e'] = [u , u'] \otimes e \w e'$ (where $u, u' \in \pi^\bullet (A)$, $e, e' \in \bigwedge^\bullet E$). Then $\pi_E^\bullet (A)$ is a $\bbz$-graded Lie algebra, and in view of the work of \cite{charCoincidence} we are led to suspect that there is often an isomorphism of the form
$$A^{\otimes n} \cong \uU (\pi_{\bbc^n}^\bullet (A) ) / I_{\bbc^n},$$
where $I_{\bbc^n}$ is some left $\uU (\pi_{\bbc^n}^\bullet (A))$-ideal with a commuting $\GL_n$-action. This is evidently true when $n = 1$, and in fact we know how to obtain such structures for $n = 2$ on the homogeneous coordinate rings of certain (isotropic) Grassmannians as quotients of $\uU (\pi_{\bbc^2}^\bullet (A) )$ (this will appear in future work). In general, the dimensions $\epsilon_i (A) := \dim_\kk \pi^i (A)$ (known as the \emph{deviations} of $A$) of a totally positive sequence determine the Hilbert function of $A$, in which case total positivity implies a numerical condition on the deviations of an arbitrary Koszul algebra. The precise relationship between Jacobi--Trudi structures, homotopy Lie algebras, and the deviations of a Koszul algebra remains mysterious to us. 

Finally, even if some examples may fail to be totally positive, we suggest that certain relaxations may still be interesting. For instance, there is a symplectic group analogue of the quadric hypersurface which is not totally positive, but does satisfy a partial total positivity, see Remark~\ref{rmk:sympletic}.

\subsection{Outline} 

Here is a quick overview of what is contained in each section.

\S\ref{sec:prelim} contains preliminaries, notation, and background references, including one of our main tools mentioned previously: Zelevinsky's functor. 

\S\ref{sec:combinatorics} discusses total positivity in the equivariant setting (the group action was ignored in the introduction). We develop the basic properties and explain some fundamental examples.

\S\ref{sec:algebraic} develops the notion of Jacobi--Trudi structure. We show under certain circumstances that this is equivalent to the notion of total positivity and also prove some parallels with Koszul algebras. In particular, this is where we prove our first main result of the paper, which realizes every (equivariant) PF-sequence as the Hilbert function of a Koszul algebra.

In \S\ref{sec:quadricSchur} we construct the Jacobi--Trudi structure for the quadric hypersurface ring using the approach with enveloping algebras outlined above. We also establish the properties of the functors $\Phi$ mentioned earlier.

Finally, in \S\ref{sec:purefree}, we apply the Jacobi--Trudi structures to construct pure free resolutions of the quadric hypersurface rings. We also examine the case of Veronese subrings of the polynomial rings.

\subsection*{Acknowledgements}

We thank Pasha Pylyavskyy, Vic Reiner, Andrew Snowden, and Jerzy Weyman for helpful discussions.

SS was supported by NSF grant DMS-2302149. KV was supported by NSF grant DMS-2202871.

\section{Preliminaries} \label{sec:prelim}

In this section we recall some background that will be essential for the rest of the paper. 

\subsection{Compositions and partitions}

  A \defi{composition} is a tuple of positive integers $(a_1,\dots,a_\ell)$, and it is a composition of $n$ if $a_1+\cdots + a_\ell=n$.
  
  A \defi{partition} is a tuple of non-negative integers $\lambda = (\lambda_1 , \dots , \lambda_n)$ with $\lambda_1 \geq \cdots \geq \lambda_n$. Its \defi{length}, denoted $\ell(\lambda)$, is the number of nonzero $\lambda_i$. We use the standard convention that two partitions are the same if they differ only in a sequence of trailing 0's. A \defi{Young diagram} of shape $\lambda$ is a left-justified diagram whose $i$th row (reading from top to bottom) has length $\lambda_i$. Given a partition $\lambda$, its \defi{transpose} $\lambda^T$ is the partition obtained by reading off the row lengths of the transposed Young diagram corresponding to $\lambda$. In a formula: $\lambda^T_i = \# \{j \mid \lambda_j \ge i\}$. We will implicitly identify partitions with their Young diagrams.

  The notation $(a^b)$ denotes a sequence of $b$ copies of $a$. Given partitions $\lambda,\mu$, we write $\mu \subseteq \lambda$ to mean $\mu_i \le \lambda_i$ for all $i$. The notation $\lambda/\mu$ denotes the skew diagram obtained by removing the Young diagram of $\mu$ from the Young diagram of $\lambda$.

\subsection{Schur functors}

   Let $D$ be a skew diagram. Define $\row (D)$ to be the tuple of integers listing off the row lengths of $D$ and $\col (D)$ the tuple of integers listing the column lengths of $D$. Given a vector space $V$, the \defi{Schur functor} $\bbs_D (V)$ is defined to be the image of a composition
   \[
     \bigwedge^{\col (D)} V \to V^{\otimes |D|} \to S^{\row (D)} (V),
   \]
   where, for an integer sequence $\alpha$, we use the shorthand $\bigwedge^\alpha V = \bigotimes_i \bigwedge^{\alpha_i} V$ and $S^\alpha V = \bigotimes_i S^{\alpha_i} V$, the first map is a tensor product of comultiplications in the exterior algebra, and the second map is a tensor product of symmetric power multiplications. The precise definition will not be necessary for the purposes of this paper; see \cite{akin1982schur} for details.

    Throughout this paper, we will be working over an algebraically closed field of characteristic $0$. That being said, the definition alluded to above is actually characteristic-free.

    The category of homogeneous degree $d$ polynomial functors admits a \defi{transpose functor} $\Omega$, which can be explicitly described on objects via
    \[
      \Omega (\bbs_\lambda) = \bbs_{\lambda^T} \quad \text{for all partitions} \ \lambda.
    \]
    In characteristic $0$, this is an exact, monoidal autoequivalence on the category of polynomial functors \cite[Proposition 7.4.3]{introtca}.

\subsection{PBW degeneration}

Let $\g$ be a Lie (super)algebra over $\bbc$. Its universal enveloping algebra is denoted $\uU(\g)$. Consider the tensor algebra $T(\g) \otimes \bbc[t]$ as a $\bbc[t]$-algebra and define $\uU_t (\g)$ to be its quotient by the two-sided ideal generated all terms of the form
    \[
      x \otimes y - y \otimes x - t [x,y] \qquad (\text{resp.} \ x \otimes y - (-1)^{|x| \cdot |y|}  y \otimes x - t[x,y])
    \]
    where $x,y \in \g$, and they are assumed to be homogeneous in the superalgebra case.
    For the following, see \cite{braverman1996poincare}.

\begin{theorem} \label{thm:PBWdegeneration}
  $\uU_t (\g)$ is a free $\bbc [t]$-module. In particular, the universal enveloping algebra $\uU (\g)$ admits a flat degeneration to the free (super)commutative algebra on $\g$.
\end{theorem}

\subsection{Zelevinsky's functor} \label{ss:zelevinsky}

Let $\g$ be a complex reductive Lie algebra. Pick a Cartan subalgebra $\fh$, Borel subalgebra $\fb$, and opposite Borel subalgebra $\fb_-$. Zelevinsky \cite{zelevinskii1987resolvents} constructed a functor from the category of locally finite $\g$-modules (i.e., $\g$-modules which are direct sums of finite-dimensional modules) to the category of vector spaces with several remarkable properties.

Namely, let $\fn_-=[\fb_-,\fb_-]$. Given a weight $\mu$ and a $\g$-representation $V$, define a functor $\Phi_{V,\mu}^\g$ from $\g$-representations to vector spaces by
\[
  \Phi_{V,\mu}^\g(M) = ((V \otimes M)/ \fn_-(V \otimes M))_\mu.
\]
where the subscript denotes taking the $\mu$-weight space.

We will only make use of the case when $\g$ is a product of $\gl_n$, so we will make the case $\gl_n$ more explicit here. In that case, we can take $\fh$ to be the subalgebra of diagonal matrices, $\fb$ to be the subalgebra of upper-triangular matrices, $\fb_-$ to be the subalgebra of lower-triangular matrices, and $\fn_-$ is then the subalgebra of strictly lower-triangular matrices.

Let $U$ be a locally finite $\gl_n(\bbc)$-representation. Given an integral weight $\alpha \in \bbz^n$, let $U_\alpha$ denote the weight space for $\alpha$, i.e., this is the span of all vectors $u \in U$ such that for all $x_1,\dots,x_n$, the diagonal matrix with entries $x_1,\dots,x_n$ acts as multiplication on $u$ by the scalar $\alpha_1x_1 + \cdots + \alpha_n x_n$. Next, given partitions $\lambda$ and $\mu$ such that $\mu \subseteq \lambda$, let $\bbs_{\lambda/\mu}(\bbc^n)$ denote the skew Schur functor and set
\[
  U[\lambda/\mu] = \hom_{\gl_n}(\bbs_{\lambda/\mu}(\bbc^n), U).
\]
If $\mu=0$, then $\bbs_\lambda(\bbc^n)$ is irreducible and this is the corresponding multiplicity space of $U$.

\begin{remark}
  In general, we have a non-canonical isomorphism
  \[
    U[\lambda/\mu] \cong \bigoplus_\nu U[\nu]^{\oplus c^\lambda_{\mu,\nu}}
  \]
  where $c^\lambda_{\mu,\nu}$ is a Littlewood--Richardson coefficient. 
\end{remark}

The symmetric group $S_n$ acts on $\bbz^n$ via permutations and we define the length of $\sigma \in S_n$ as
\[
  \ell(\sigma) = \#\{ i < j \mid \sigma(i) > \sigma(j) \}.
\]
Finally, we define a special weight $\rho = (n-1,n-2,\dots,1,0)$, and define the dotted action of $S_n$ on weights by
\[
  \sigma \bullet \alpha = \sigma(\alpha + \rho) - \rho.
\]

\begin{theorem}[Zelevinsky]
  Given $U$, and dominant integral weights $\lambda$, $\mu$, there exists an acyclic complex ${\bf F}(U)^{\lambda,\mu}_\bullet$, functorial in $U$, with terms
\[
  {\bf F}(U)^{\lambda,\mu}_i = \bigoplus_{\substack{\sigma \in S_n\\ \ell(\sigma)=i}} U_{\lambda - \sigma\bullet \mu}
\]
and
\[
  {\rm H}_0({\bf F}(U)^{\lambda,\mu}_\bullet) = U[\lambda/\mu].
\]
\end{theorem}

In the notation of \cite{zelevinskii1987resolvents}, the complex ${\bf F}(U)^{\lambda,\mu}_\bullet$ is obtained by applying the functor $\Phi_{U,\lambda}^\g$ to the BGG resolution of $\bbs_\mu(\bbc^n)$.

Since the construction is functorial, if an algebra $R$ acts on $U$ and commutes with the $\gl_n$-action, then ${\bf F}(U)^{\lambda,\mu}_\bullet$ is a complex of $R$-modules. This will be important for our applications.

\begin{remark}
  In \cite{zelevinskii1987resolvents}, Zelevinsky applies this construction to the symmetric algebra $U = S^\bullet(V \otimes \bbc^n)$. In that case, we have $U_\alpha = S^{\alpha_1}(V) \otimes \cdots \otimes S^{\alpha_n}(V)$ and $U[\lambda/\mu] \cong \bbs_{\lambda/\mu}(V)$. By taking the equivariant Euler characteristic of ${\bf F}(U)^{\lambda,\mu}_\bullet$, we get an expression for the skew Schur polynomial in terms of products of complete homogeneous symmetric polynomials, and this coincides with the classical Jacobi--Trudi identity.

  Heuristically, we can think of this construction as transforming the Cauchy identity (which $S^\bullet(V \otimes \bbc^n)$ encodes) into the Jacobi--Trudi identity, and this is how we will generalize it.

  This particular example was also considered by Akin in \cite{akin1988complexes,akin1992complexes}. Other examples were recently used in \cite{JTdet-var}.
\end{remark}

\section{Combinatorial aspects} \label{sec:combinatorics}

Let $G$ be a group and $A_1,A_2,\dots$ a sequence of finite-dimensional $G$-representations over an algebraically closed field ${\bf k}$, set $A_0 = {\bf k}$ to be the trivial representation, and set $A = \bigoplus_{d \ge 0} A_d$. We will keep this notation throughout the section.

We will work in the Grothendieck ring ${\rm K}(G)$ of finite-dimensional $G$-representations with the product induced by tensor product. Our cases of interest are when $G$ is either finite or a reductive algebraic group (and the $A_i$ are algebraic representations) in which case we can take ${\rm K}(G)$ to be the corresponding ring of characters. The class of a representation $W$ is denoted $[W] \in {\rm K}(G)$.

\subsection{Equivariant total positivity}

We begin our discussion by formulating an equivariant analogue of the notion of total positivity. These facts are likely well-known, but we could not find a convenient reference, so we will provide proofs when needed.

Consider the ${\rm K}(G)$-valued matrix $T_A$ whose rows and columns are indexed by the set of positive integers and whose entries are defined by
\[
  (T_A)_{i,j} = \begin{cases} [A_{i-j}] & \text{if $i\ge j$} \\ 0 & \text{else} \end{cases}.
\]
We call $T_A$ the \defi{Toeplitz matrix} of $A$. We say that $A$ is \defi{$G$-totally positive}, and that $([A_i])$ is a \defi{(normalized) $G$-P\'olya frequency sequence} (or \defi{normalized $G$-PF sequence}), if every square submatrix of $T_A$ is a non-negative element of ${\rm K}(G)$ (i.e., either 0 or represented by the class of a representation). Here, ``normalized'' refers to the fact that we require $A_0$ to be the trivial representation; we will drop the adjective normalized from now on since we will not discuss any other kind of PF-sequences. Note that by forgetting the $G$-action, we see that if $A$ is $G$-totally positive, then the sequence $(\dim A_i)_{i \ge 0}$ is a PF-sequence in the usual sense (see, for instance, \cite{brenti} for an introduction).

There is a universal setting for this in which $[A_1],[A_2],\dots$ are algebraically independent. Namely, let $\Lambda$ be the ring of symmetric functions in variables $x_1,x_2,\dots$ with integer coefficients, and let
\[
  h_d = \sum_{i_1\le\cdots\le i_d} x_{i_1}\cdots x_{i_d}
\]
be the complete homogeneous symmetric function of degree $d$. (See \cite[\S 7.5]{stanley}.) Then $\Lambda$ is a polynomial ring $\bbz[h_1,h_2,\dots]$ \cite[Corollary 7.6.2]{stanley} and we have a ring homomorphism
\[
  \psi \colon \Lambda \to {\rm K}(G), \qquad \psi(h_d) = [A_d].
\]
Define the matrix ${\bf T}$ by
\[
  {\bf T}_{i,j} = \begin{cases} h_{i-j} & \text{if $i\ge j$} \\ 0 & \text{else} \end{cases}.
\]

Given column indices $I = (i_1 < \cdots < i_r)$ and row indices $J = (j_1 < \cdots < j_r)$, let ${\bf T}_{J,I}$ be the corresponding submatrix. Define two partitions by $\lambda = (i_r-r,\dots,i_1-1)$ and $\mu = (j_r-r,\dots,j_1-1)$. Then by \cite[\S 7.16]{stanley}, we have
\[
  \det({\bf T}_{J,I}) = s_{\lambda/\mu}(x) = \sum_\nu c^{\lambda}_{\mu,\nu} s_\nu(x),
\]
where the $c^{\lambda}_{\mu,\nu}$ are Littlewood--Richardson coefficients, which are non-negative integers \cite[\S 7.A1.3]{stanley}. In particular, the minors of ${\bf T}$, or $T_A$ in general, can be written as non-negative integer linear combinations of minors where the row indices are always taken to be $1,\dots,r$.

Motivated by this discussion, given partitions $\lambda,\mu$ each having at most $r$ parts, we define
\[
  s^A_{\lambda/\mu} = \det ( [A_{\lambda_i-\mu_j-i+j}] )_{i,j=1}^r.
\]
This is independent of the choice of $r$; furthermore, the discussion above can be summarized as saying that $A$ is a $G$-PF sequence if and only if $s_\lambda^A \ge 0$ for all partitions $\lambda$.

In what follows, $\lambda/\mu$ is a horizontal strip if each column contains at most 1 box and the elementary symmetric functions are defined by
\[
  e_d = \sum_{1 \le i_1 < i_2 < \cdots < i_d} x_{i_1} x_{i_2} \cdots x_{i_d}.
\]

\begin{prop}\label{cor:fundamentalIdentities}
    Suppose $A$ is $G$-totally positive. Then, any symmetric function identity for the Schur polynomials $s_\lambda \in \Lambda$ must also hold for the K-classes $s_{\lambda}^A \in {\rm K}(G)$. In particular:
    \begin{enumerate}
        \item (Generalized Pieri rule) For all partitions $\lambda$ and integers $d \geq 1$, we have
          \[
            s_\lambda^A \cdot [A_d] = \bigoplus_{\substack{\mu \ \text{partition} \\ |\mu| - |\lambda| = d, \\
                \mu / \lambda \ \text{horizontal strip}}} s_\mu^A.
          \]

        \item (Transpose duality) If $n \ge \lambda_1$, 
          \[
            s_{\lambda / \mu}^A = \det ( \psi(e_{\lambda_i^T-\mu_j^T - i + j}))_{i,j=1}^n,
          \]
          where $e_d \in \Lambda$ denotes the $d$th elementary symmetric function.

        \item (Generalized Littlewood--Richardson rule) For any partitions $\mu \subseteq \lambda$, we have
          \[
            s_{\lambda / \mu}^A = \sum_{\nu} c_{\lambda,\mu}^\nu s_{\nu}^A,
          \]
        where $c_{\lambda,\mu}^\nu$ denotes the Littlewood--Richardson coefficient.
    \end{enumerate}
\end{prop}

\begin{proof}
  These follow from the corresponding facts about symmetric functions. For (1), see \cite[Theorem 7.15.7]{stanley}, for (2), see \cite[Corollary 7.16.2]{stanley}, and for (3), see \cite[\S 7.A1.3]{stanley}.
\end{proof}

It follows from the Pieri rule that if $s_\lambda^A=0$ and $\lambda \subseteq\mu$, then $s_\mu^A=0$. In particular, the set $\{\lambda \mid s_\lambda^A=0\}$ is an order ideal in the set of partitions ordered by containment (i.e., Young's lattice). In the special case when this ideal is of the form $\{ \lambda \mid \lambda_{r+1} > s\}$ for some integers $r,s$, we will say that $A$ has \defi{Schur dimension $r|s$}. The reasoning for this notation comes from the study of superalgebras, as will be explained in Example~\ref{ex:schur-dim}.

\subsection{Operations preserving positivity}

Let $d$ be a fixed positive integer and define a new sequence of $G$-representations $A_{(d)}$ by $(A_{(d)})_i = A_{di}$. We will call this the \defi{$d$th Veronese sequence}.

\begin{prop} \label{prop:veronese}
  If $A$ is $G$-totally positive, then so is $A_{(d)}$. More precisely, if $\ell(\lambda), \ell(\mu) \le r$, then we have
  \[
    s_{\lambda/\mu}^{A_{(d)}} = s_{\alpha/\beta}^A
  \]
  where
  \[
    \alpha_i = d\lambda_i + (d-1)(r-i), \qquad \beta_j = d\mu_j + (d-1)(r-i).
  \]
\end{prop}

\begin{proof}
  This essentially follows from the definitions: $s_{\lambda/\mu}^{A_{(d)}}$ is the determinant of the submatrix of $T_{A_{(d)}}$ with columns $\lambda_r + 1 < \cdots < \lambda_1 + r$ and rows $\mu_r + 1 < \cdots < \mu_1 + r$, which, by definition, is the submatrix of $T_A$ with columns $d(\lambda_r+1) < \cdots < d(\lambda_1+r)$ and rows $d(\mu_r+1) < \cdots < d(\mu_1 + r)$. Its determinant is $s_{\alpha'/\beta'}^A$ where $\alpha'_i = \alpha_i+d-1$ and $\beta'_j = \beta_j + d-1$. But subtracting the same amount from all entries of $\alpha'$ and $\beta'$ does not affect the determinant; we have done this for notational simplicity.
\end{proof}

Now suppose that $A$ and $B$ are two sequences of $G$-representations. Define the \defi{tensor product sequence} by
\[
  (A \otimes B)_n = \bigoplus_{i=0}^n A_i \otimes B_{n-i}.
\]

\begin{prop}
  If $A$ and $B$ are $G$-totally positive, then so is $A \otimes B$. More precisely, we have
  \[
    s_{\lambda/\mu}^{A\otimes B} = \sum_{\mu \subseteq \nu \subseteq \lambda} s_{\lambda/\nu}^A s_{\nu/\mu}^B.
  \]
\end{prop}

\begin{proof}
  First, we have $T_{A\otimes B} = T_A T_B = T_BT_A$.
  Pick $r \ge \ell(\lambda), \ell(\mu)$ and set $I = (\lambda_r+1 < \cdots < \lambda_1 + r)$ and $J = (\mu_r+1 < \cdots < \mu_1+r)$. The Cauchy--Binet formula gives:
  \[
    s_{\lambda/\mu}^{A \otimes B} = \det(T_{A \otimes B})_{J,I} = \det(T_B T_A)_{J,I} = \sum_K \det(T_B)_{J,K} \det(T_A)_{K,I}
  \]
  where the sum is over all subsets $K$ of size $r$. Each $K$ translates to some partition $\nu$; and the product of determinants is nonzero if and only if $\mu\subseteq \nu \subseteq \lambda$.
\end{proof}

\subsection{Examples}

We conclude this section with some examples illustrating the notions discussed so far.

\begin{example}\label{ex:segreNonpositive}
  Consider the universal case where $A = S^\bullet(V)$ and $B = S^\bullet(W)$ and consider their Segre (or Hadamard product) $C$ with $C_d = A_d \otimes B_d = S^d(V) \otimes S^d(W)$ with the action of $G = \GL(V) \times \GL(W)$. This is not $G$-totally positive since $s_{2,2,2}^C$ is not Schur-positive:
  \begin{align*}
    s_{2,2,2}^C &= \det \begin{pmatrix} 
h_{2} (x) h_{2} (y) & h_{3} (x) h_{3} (y) & h_{4} (x) h_{4} (y)\\
      h_{1} (x) h_{1} (y) & h_{2} (x) h_{2} (y) & h_{3}(x) h_{3} (y)\\
      1 & h_{1} (x) h_{1} (y) & h_{2} (x)h_{2} (y) \end{pmatrix}\\
    &=
      s_{2,2,2}(x)s_{6} (y)+\left(-s_{3,3}(x)+2s_{2,2,2} (x)\right)s_{5,1} (y)\\
                &\quad +\left(2s_{4,2} (x)+2s_{3,2,1} (x)+3s_{2,2,2} (x) \right)s_{4,2}(y)
                  +\left(-s_{3,3} (x)+s_{2,2,2} (x)\right)      s_{4,1,1} (y)\\
                &\quad +\left(-s_{5,1} (x)-s_{4,1,1} (x)+s_{2,2,2} (x)\right)s_{3,3} (y)
                  +\left(2s_{4,2} (x)+2s_{3,2,1}(x)+2s_{2,2,2} (x)\right)s_{3,2,1} (y)\\
                &\quad +\left(s_{6} (x)+2s_{5,1} (x)+3s_{4,2} (x)+s_{4,1,1} (x)+s_{3,3} (x)+2s_{3,2,1} (x)+s_{2,2,2} (x)\right)s_{2,2,2} (y).
  \end{align*}

  It is worth noting that if we take $G$ to be trivial, then $(\dim C_i)_{i \ge 0}$ \emph{is a PF sequence} \cite[Theorem 0.2]{wagner}. In particular, a PF sequence need not be totally positive at the level of characters for some nontrivial group action.
\end{example}

\begin{example} \label{ex:schur-dim}
  Let $V$ be a vector superspace with $\dim V_0=r$ and $\dim V_1=s$ and let $A_n = S^n(V)$ in the super sense (in ordinary notation, $A_n = \bigoplus_{i=0}^n S^i(V_0) \otimes \bigwedge^{n-i}(V_1)$) and take $G = \GL(V)$, the general linear supergroup. Then $s_\lambda^A$ is the hook Schur function in \cite[\S 6]{bereleregev}. The explicit formula given there in terms of semistandard tableaux shows that $s_\lambda^A=0$ if and only if $\lambda_{r+1} > s$, so that $A$ has Schur dimension $r|s$.
\end{example}

\begin{example}
  Let $V$ be a vector space and $A_n = V^{\otimes n}$ ($G$ will not be relevant here). Then $A$ has Schur dimension $1|0$ since $s_{1,1}^A = \det\begin{pmatrix} [V] & [V]^2 \\ 1 & [V] \end{pmatrix} = 0$.
\end{example}

\begin{example} \label{ex:ver-schurdim}
  Let $V$ be an $r$-dimensional vector space, fix a positive integer $d$, and $A_n = S^{nd} (V)$, which is the $d$th Veronese sequence of $A=S^\bullet (V)$ from Example~\ref{ex:schur-dim}. Take $G = \GL(V)$. We know that $s_\lambda^A = s_{\alpha/\beta}(x_1,\dots,x_r)$ where the latter is a skew Schur polynomial in $r$ variables and $\alpha,\beta$ are determined by Proposition~\ref{prop:veronese}. From its interpretation as a sum over semistandard Young tableaux \cite[\S 7.10]{stanley}, $s_{\alpha/\beta}(x_1,\dots,x_r) \ne 0$ if and only if each column of the skew diagram $\alpha/\beta$ contains at most $r$ boxes. This fails exactly when  $d\lambda_{r+1} > (d-1)r$.

  This shows that $A$ has Schur dimension $r|c$ where $c$ is the smallest integer such that $c+1 > r - \frac{r}{d}$. In particular, there are precisely three cases when $c \le 1$: (1) $d = 1$; (2) $r=2$; (3) $r=3$ and $d=2$. These are exactly the cases when the Veronese algebra $A$ has finite Cohen--Macaulay type.
\end{example}

\section{Algebraic aspects}  \label{sec:algebraic}

In this section, we will formulate an algebraic notion of positivity in terms of Jacobi--Trudi structures. This notion will turn out to be easier to work with for our purposes, but is actually equivalent to $G$-positivity, a fact which we will prove using Zelevinsky's functor. We also take a closer look at the relationship between $G$-positivity and Koszulness, establishing the fact that \emph{every} $G$-PF sequence is the equivariant Hilbert function of some $G$-equivariant Koszul algebra (see Theorem \ref{thm:GPFalgebraStructure}). 

\subsection{Jacobi--Trudi structures}\label{sec:JTstructures}

We continue to use the notation from the beginning of Section~\ref{sec:combinatorics}. The notion of $G$-total positivity is combinatorial, but our intended applications of this concept are algebraic. The next definition turns out to be equivalent to $G$-total positivity (at least when $G$ acts semisimply on tensor powers of $A$), but will allow for algebraic enhancements, which will be discussed later.

\begin{definition}[Jacobi--Trudi structure]\label{def:JTAlgebras}
  A \defi{Jacobi--Trudi structure} on $A$ is a $\GL_n$-action on $A^{\otimes n}$ for every $n > 0$ satisfying:
    \begin{enumerate}
        \item For every weight $\lambda \in \bbz^n$, we have
        $$(A^{\otimes n})_\lambda = A_{\lambda_1} \otimes A_{\lambda_2} \otimes \cdots \otimes A_{\lambda_n}.$$
        \item For every $d$, the action of $\GL_n$ on $A^{\otimes n}$ extends the action of $\GL_{n-1}$ on $A^{\otimes n-1} \otimes A_d \subset A^{\otimes n}$ where $\GL_{n-1}$ acts trivially on $A_d$. Here we embed $\GL_{n-1}$ into $\GL_n$ via $x \mapsto \begin{bmatrix} x & 0 \\ 0 & 1\end{bmatrix}$.
        \item $\GL_n$ commutes with the diagonal action of $G$ on $A^{\otimes n}$. \qedhere
      \end{enumerate}
    \end{definition}

\begin{remark}
  In our examples, $A$ will have an algebra structure, but we will not assume that the $\GL_n$-action is compatible with this algebra structure in any way, and in fact generally it will not be.
  \end{remark}

\begin{example}\label{ex:symJTstructure}
    The most fundamental example of a Jacobi--Trudi structure is given by the symmetric algebra on any vector superspace $V$ with $G=\GL(V)$; the Jacobi--Trudi structure comes from the isomorphisms
    \[
      S^\bullet (V \otimes \bbc^n) \cong S^\bullet (V)^{\otimes n}. \qedhere
    \]
\end{example}

\begin{example}\label{ex:tensorJTstructure}
    The tensor algebra $T^\bullet (V)$ with the natural action of $G=\GL(V)$ is another example of an algebra admitting a Jacobi--Trudi structure in an even simpler way. Consider the Segre product
    $$T^\bullet (V) \circ S^\bullet (\bbc^n) := \bigoplus_{d \geq 0} T^d (V) \otimes S^d (\bbc^n),$$
    where the generators $V \otimes \bbc^n$ are considered to have degree $1$. Then we claim that there is an isomorphism of $\bbc$-vector spaces
    $$T^\bullet (V) \circ S^\bullet (\bbc^n) \cong (T^\bullet (V))^{\otimes n},$$
    satisfying the properties $(1) - (3)$. First, note that the property $(1)$ translates into saying that for every weight $\lambda$ we have an isomorphism
    $$(T^\bullet(V) \circ S^\bullet (\bbc^n))_\lambda \cong T^{\lambda_1} (V) \otimes \cdots \otimes T^{\lambda_n} (V) \cong T^{|\lambda|} (V).$$
    However, the symmetric algebra has the property that for every $\lambda \in \bbz^n_{\ge 0}$ there is a unique basis vector of $S^\bullet (\bbc^n)$ with weight $\lambda$, so $(1)$ follows immediately since there is an isomorphism
    $$(T^\bullet (V) \circ S^\bullet (\bbc^n))_\lambda = T^{|\lambda|} (V) \otimes  \underbrace{x_1^{\lambda_1} \cdots x_n^{\lambda_n}}_{\in S^{|\lambda|} (\bbc^n)}.$$
    The property $(2)$ is even simpler since the induced $\GL_{n-1}$-action comes from the natural inclusion $S^\bullet (\bbc^{n-1}) \hookrightarrow S^\bullet (\bbc^n)$. 
\end{example}

\subsection{Jacobi--Trudi complexes}

Our main motivation for introducing the concept of a Jacobi--Trudi structure is that it gives a series of acyclic chain complexes whose Euler characteristics are the minors of the associated Toeplitz matrix $T_A$. To do this, we utilize Zelevinsky's functor from \S\ref{ss:zelevinsky}.

Suppose that $A$ has a Jacobi--Trudi structure. Let $\lambda,\mu$ be partitions of length at most $n$. We define
\[
  {\bf F}^{A,\lambda,\mu}_\bullet = {\bf F}(A^{\otimes n})^{\lambda,\mu}_\bullet,   \qquad \bbs_{\lambda/\mu}^A = (A^{\otimes n})[\lambda/\mu].
\]
Then we have
\begin{align*}
  {\bf F}^{A,\lambda,\mu}_i &=  \bigoplus_{\substack{\sigma \in S_n\\ \ell(\sigma)=i}} (A^{\otimes n})_{\lambda - \sigma \bullet \mu} = \bigoplus_{\substack{\sigma \in S_n\\ \ell(\sigma)=i}} A_{\lambda_1 - (\sigma \bullet \mu)_1} \otimes \cdots \otimes A_{\lambda_n - (\sigma \bullet \mu)_n}.
\end{align*}
Furthermore, ${\bf F}^{A,\lambda,\mu}_\bullet$ is acyclic and we have an isomorphism
\[
  {\rm H}_0({\bf F}^{A,\lambda,\mu}_\bullet) \cong \bbs_{\lambda/\mu}^A.
\]

\begin{example}
  The classical Cauchy identity \cite[\S 7.12]{stanley} reads
    $$S^\bullet (V \otimes \bbc^n) = \bigoplus_{\substack{\lambda \ \text{partition} \\ \ell (\lambda) \leq n }} \bbs_\lambda (V) \otimes \bbs_\lambda (\bbc^n),$$
    whence the skew Schur functors $\bbs_{\lambda/\mu}^{S^\bullet (V)}$ are the classically defined Schur functors $\bbs_{\lambda/\mu}(V)$.

    For the tensor algebra, the only nonzero Schur functors $\bbs_\lambda^{T^\bullet(V)}$ are when $\lambda = (d)$ has one row. In particular, 
    \[
      \bbs_{\lambda/\mu}^{T^\bullet (V)} = \begin{cases}
        T^{|\lambda|-|\mu|} (V) & \text{if $\lambda/\mu$ has at most one box per column} , \\
        0 & \text{otherwise}.
      \end{cases} \qedhere
      \]
\end{example}

\begin{prop} \label{prop:JT-det}
  The Euler characteristic of ${\bf F}^{A,\lambda,\mu}_\bullet$ can be expressed as an $n \times n$ determinant
  \[
    \chi({\bf F}^{A,\lambda,\mu}_\bullet) = \det ( [A_{\lambda_i-\mu_j - i + j}])_{i,j=1}^n = [\bbs_{\lambda/\mu}^A] = s^A_{\lambda/\mu},
  \]
  where we interpret $[A_d]=0$ if $d<0$.

  In particular, if $A$ admits a Jacobi--Trudi structure, then it is $G$-totally positive and 
  \[
    \dim (\bbs_{\lambda/\mu}^A) = \det ( \dim(A_{\lambda_i-\mu_j - i + j}) )_{i,j=1}^n.
  \]
\end{prop}

\begin{proof}
  Let $\sigma \in S_n$ be any permutation. Notice that by definition of the dot action, there is an equality 
  $$\sigma \bullet \mu = (\mu_{\sigma^{-1} (1)} + 1 - \sigma^{-1} (1) , \dots , \mu_{\sigma^{-1} (n)} +n - \sigma^{-1} (n)).$$
  With this in mind, set $x_{i,j} := [A_{\lambda_i - \mu_j + j -i}]$. By definition of the determinant there is an equality
  \begingroup\allowdisplaybreaks
  \begin{align*}
      \det (x_{i,j})_{i,j=1}^n &= \sum_{\sigma \in S_n} (-1)^{\ell (\sigma)} x_{1, \sigma^{-1} (1)} x_{2 , \sigma^{-1} (2)} \cdots x_{n, \sigma^{-1} (n)} \\
      &= \sum_{\sigma \in S_n} (-1)^{\ell (\sigma)} [A_{\lambda_1 -\mu_{\sigma^{-1} (1)} + \sigma^{-1} (1) -1}] [A_{\lambda_2 - \mu_{\sigma^{-1} (2)} +\sigma^{-1} (2) - 2}] \cdots [ A_{\lambda_n - \mu_{\sigma^{-1} (n)} +\sigma^{-1} (n) - n}] \\
      &= \sum_{\sigma \in S_n} (-1)^{\ell (\sigma)}[(A^{\otimes n})_{\lambda - \sigma \bullet \mu} ]. \qedhere
  \end{align*}
  \endgroup
\end{proof}
A priori, this construction depends on $n$, but this dependence can be removed.

\begin{prop}\label{prop:schurStability}
  If $n \ge \ell(\lambda),\ell(\mu)$, then we have a natural isomorphism
  \[
    A^{\otimes n}[\lambda/\mu] \cong A^{\otimes (n+1)}[\lambda /\mu].
  \]
\end{prop}

\begin{proof}
  We treat $\bbc^n$ as a subspace of $\bbc^{n+1}$ via $x \mapsto (x,0)$. Then there is a canonical inclusion $\bbs_{\lambda/\mu}(\bbc^{n}) \subset \bbs_{\lambda/\mu}(\bbc^{n+1})$.  Given a $\GL_{n+1}$-equivariant map $f \colon \bbs_{\lambda/\mu}(\bbc^{n+1}) \to A^{\otimes (n+1)}$, we restrict it to this subspace to get a $\GL_n$-equivariant map $\bbs_{\lambda/\mu}(\bbc^n) \to A^{\otimes (n+1)}$. The diagonal matrix in $\GL_{n+1}$ with entries $1,\dots,1,x$ acts trivially on $\bbs_{\lambda/\mu}(\bbc^n)$, and hence the image is contained in $A^{\otimes n} \otimes A_0$, which we identify with $A^{\otimes n}$. Thus, restriction defines a map
  \[
    A^{\otimes (n+1)}[\lambda/\mu] \to A^{\otimes n}[\lambda/\mu].
  \]
  This map is injective since $\bbs_{\lambda/\mu}(\bbc^n)$ generates $\bbs_{\lambda/\mu}(\bbc^{n+1})$ as a $\GL_{n+1}$-representation. It follows from Proposition~\ref{prop:JT-det} that $\dim A^{\otimes n}[\lambda/\mu] =  \dim A^{\otimes (n+1)}[\lambda/\mu]$, so the map above is an isomorphism.
\end{proof}

The following is immediate from the definitions above.

\begin{prop}[Generalized Cauchy identity]\label{lem:JTcauchyIdentity}
    Assume $A$ has a Jacobi--Trudi structure. Then there is a $\GL_n$-equivariant decomposition
    $$A^{\otimes n} \cong \bigoplus_{\substack{\lambda \ \text{partition} \\ \ell (\lambda) \leq n }} \bbs_\lambda^A \otimes \bbs_\lambda (\bbc^n).$$
    Furthermore, this decomposition can be chosen to be $G$-equivariant if we assume that $G$ acts semisimply on $A^{\otimes n}$.
\end{prop}

\begin{prop}\label{prop:JTandSchurEquivalence}
    Assume that $G$ acts semisimply on each tensor power $A^{\otimes n}$. Then the following are equivalent:
    \begin{enumerate}
        \item $A$ admits a Jacobi--Trudi structure.
        \item $A$ is $G$-totally positive.
    \end{enumerate}
\end{prop}

\begin{proof}
  $(1) \implies (2)$: This follows from Proposition~\ref{prop:JT-det}.

  $(2) \implies (1)$: Let $\Lambda(x)$ be the ring of symmetric functions in $x_1,x_2,\dots$ and let $\Lambda(y)$ be the ring of symmetric polynomials in $y_1,\dots,y_n$. Consider the map
  \[
    \psi \colon \Lambda(x) \otimes \Lambda(y) \to {\rm K}(G) \otimes \Lambda(y), \qquad h_d(x) \otimes f(y) \mapsto [A_d] \otimes f(y).
  \]
  By the Cauchy identity, we have 
  \[
    \sum_\lambda s_\lambda(x) s_\lambda(y) = \prod_{i=1}^n \prod_{j \ge 1} \frac{1}{1-x_jy_i} = \prod_{i=1}^n \left( \sum_{d \ge 0} h_d(x) y_i^d \right).
  \]
  Applying $\psi$ to this identity gives
  \[
    \sum_\lambda s_\lambda^A s_\lambda(y) = \prod_{i=1}^n \left( \sum_{d \ge 0} [A_d] y_i^d \right).
  \]
  Now let $\bbs_\lambda^A$ be any $G$-representation such that $[\bbs_\lambda^A]=s_\lambda^A$ (by assumption, such a representation exists). Then the identity above shows that $\bigoplus_\lambda \bbs_\lambda^A \otimes \bbs_\lambda(\bbc^n) \cong A^{\otimes n}$ as $G$-representations, and the left side is evidently a representation of $\GL_n$.
\end{proof}

\subsection{Koszul algebras and Jacobi--Trudi structures}

This section requires more Lie theory background than the rest of the paper, but it is also independent of the rest of the paper, so can be skipped without loss of continuity.

Let $\kk$ be a field of characteristic 0. A non-negatively-graded $\kk$-algebra $A$ is \defi{Koszul} if the residue field $\kk$ has a linear minimal free resolution over $A$, i.e., $\Tor^A_i(\kk,\kk)$ is concentrated in (internal) degree $i$ for all $i \ge 0$. If $G$ is a group acting on $A$, its equivariant Hilbert function is the sequence $([A_i])_{i \geq 0}$. In this section, we clarify the relationship between equivariant Hilbert functions of Koszul algebras and $G$-PF sequences.

Given a Koszul algebra, one can form the quadratic dual $A^! \coloneq \ext_A^\bullet (\kk , \kk)$; this algebra can be described more concretely as a quotient of the tensor algebra by the quadratic relations ``orthogonal'' to the quadratic relations defining $A$ (see \cite{polishchuk2005quadratic} for a rigorous introduction to Koszul algebras). 

\begin{notation}
  Let $\g := \gl_n$, let $\fh$ be the subalgebra of diagonal matrices, and let $\fb$ be the subalgebra of upper-triangular matrices. Then $\fh$ is isomorphic to a quotient of $\fb$, so any representation of $\fh$ can be pulled back to a representation of $\fb$. For $w \in S_n$, the Verma module for the weight $w \bullet(0) = w \cdot \rho - \rho$ (with $\rho = (n-1,n-2,\dots,1,0)$) is
  \[
    M^\g(w) = \uU(\g) \otimes_{\uU (\fb)} \bbc_{w \bullet (0)}.
  \]
  The notation $\BGG_n$ will denote the BGG resolution \cite{bernstein1971structure} of $\bbc$ as a $\uU (\g)$-module. We have
  \[
    (\BGG_n)_i = \bigoplus_{\substack{w \in S_n\\ \ell(w)=i}} M^\g(w).
  \]
  More generally, for any composition ${\bf a}=(a_1,\dots,a_\ell)$ of $n$, we let $\g_{\bf a} = \gl_{a_1} \times \cdots \times \gl_{a_\ell}$ denote the block diagonal matrices with block sizes $a_1,\dots,a_\ell$ and let $\fp_{\bf a}$ denote the block upper-triangular matrices. Then $\g_{\bf a}$ is a quotient Lie algebra of $\fp_{\bf a}$, so any representation $V$ of $\g_{\bf a}$ pulls back to one of $\fp_{\bf a}$. We let
  \[
    \ind_{\fp_{\bf a}}^{\g} V = \uU(\g) \otimes_{\uU(\fp_{\bf a})} V. \qedhere
  \]
 \end{notation}
  
  By transitivity of base change, for any $w=(w_1,\dots,w_\ell) \in S_{a_1}\times \cdots \times S_{a_\ell}$, we have
  \[
    \ind_{\fp_{\bf a}}^{\gl_n}( M^{\gl_{a_1}}(w_1) \otimes \cdots \otimes M^{\gl_{a_\ell}}(w_\ell)) \cong M^{\gl_n}(w). \qedhere
  \]

\begin{lemma}[Akin] \label{lem:akin}
    Given any composition ${\bf a}=(a_1 , \dots , a_\ell)$ of $n$, there is a subcomplex $\BGG_{{\bf a}} \subset \BGG_n$ of $\uU(\g_{\bf a})$-modules obtained by restricting to the sum of the terms $M^\g(w)$ where $w \in S_{a_1} \times \cdots \times S_{a_\ell} \subset S_n$. 

    There is moreover an isomorphism of complexes
    \[
      \BGG_{{\bf a}} \cong \ind_{\fp_{\bf a}}^{\g} (\BGG_{a_1} \otimes \cdots \otimes \BGG_{a_\ell}).
  \]
\end{lemma}

\begin{remark}
    The complexes $\BGG_{{\bf a}}$ are necessarily resolutions of $\bbc$ over the subalgebra $\uU (\g_{\bf a}) \subset \uU(\g)$. The isomorphism in the previous result is largely affected by sign conventions for tensor products of complexes and the fact that the BGG complex is determined up to some choice of orientation function on the Bruhat poset. As proved by Akin \cite[Proof of Theorem 2]{akin1992complexes}, there is a choice of signs for which there is an honest \emph{equality} (modulo canonical isomorphisms)
    \[
      \BGG_{{\bf a}} = \ind_{\fp_{\bf a}}^{\g} (\BGG_{a_1} \otimes \cdots \otimes \BGG_{a_\ell}), \quad \text{for all  $(a_1 , \dots , a_\ell)$}.
    \]
    We will tacitly assume for the remainder of this section that the signs chosen on the BGG resolution align with those of Akin, and hence satisfy the equality above.
\end{remark}

\begin{lemma} \label{lem:JT-tensor}
  Assume $A$ is a $G$-PF sequence. Let ${\bf a}=(a_1,\dots,a_\ell)$ be a composition of $n$.
Define $b_i = \sum_{j=1}^i a_j$ and for each $i=1,\dots,\ell$, define the length $a_i$ sequence $\mu(i) = (\lambda_{b_{i-1}+1}, \lambda_{b_{i-1}+2}, \dots,\lambda_{b_i})$. Then Zelevinsky's functor induces an isomorphism of complexes
\[
  \Phi_{A^{\otimes n}, \lambda}^{\gl_n} (\BGG_{{\bf a}} )      \cong
  \Phi^{\gl_{a_1}}_{A^{\otimes a_1}, \mu(1)} (\BGG_{a_1}) \otimes \cdots \otimes \Phi^{\gl_{a_\ell}}_{A^{\otimes a_\ell}, \mu(\ell)} (\BGG_{a_\ell})
    \]
    where $A^{\otimes n}$ is viewed as a representation of $\g_{\bf a}=\gl_{a_1} \times \cdots \times \gl_{a_\ell}$.
\end{lemma}

\begin{proof}
  Pick $w = (w_1,\dots,w_\ell) \in S_{a_1} \times \cdots \times S_{a_\ell}$.
  Let $\fn_-$ be the subalgebra of strictly lower-triangular matrices in $\gl_n$ and let $\fn'_-$ be the tuples of lower-triangular matrices in $\g_{\bf a}$. First note that by basic properties of the induction functor, for any $\g_{\bf a}$-representation $V$, we have
  \[
    \ind_{\fp_{\bf a}}^{\gl_n}(V)/ \fn_-\ind_{\fp_{\bf a}}^{\gl_n}(V) = V / \fn'_- V.
  \]
In particular, 
  \begin{align*}
    \Phi^{\gl_n}_{A^{\otimes n}, \lambda}( M^{\gl_n}(w))
    &= \Phi^{\gl_n}_{A^{\otimes n}, \lambda}(\ind_{\fp_{\bf a}}^{\gl_n} (M^{\gl_{a_1}}(w_1) \otimes \cdots \otimes M^{\gl_{a_\ell}}(w_\ell))\\
    &= \Phi^{\gl_{a_1}}_{A^{\otimes a_1}, \mu(1)}( M^{\gl_{a_1}}(w_1)) \otimes \cdots \otimes \Phi^{\gl_{a_\ell}}_{A^{\otimes a_\ell}, \mu(\ell)} (M^{\gl_{a_\ell}}(w_\ell))
  \end{align*}
  Thus at the level of modules, Zelevinsky's functor induces an isomorphism of the desired form, so it remains to show that this is compatible with the differentials, but this is an immediate consequence of Lemma~\ref{lem:akin}.
\end{proof}

Let $A$ be $G$-totally positive. The graded dual of a $G$-PF sequence is also a $G$-PF sequence, so we can consider the induced Jacobi--Trudi structure on $(A^*)^{\otimes n}$ for $n \geq 1$. Consider the complex:
    \[
      \Phi_{(A^*)^{\otimes 2}, (1,1)} (\BGG_2) = \left( 0 \to A^*_2 \to A_1^* \otimes A_1^* \right).
    \]
    This complex induces a map $\Delta_{1,1} \colon A_2^* \to A_1^* \otimes A_1^*$, and for general $a,b>0$ we construct maps
    \[
      \Delta_{a,b} \colon A_{a +b}^* \to A_a^* \otimes A_b^*, \qquad      \Delta_{a , b} \coloneq   (-1)^{a+b} \pi_{a,b} \circ d^{JT^{(1^{a+b})}}_{a+b-1}
    \]
    where $JT^{(1^{a+b})} \coloneq \Phi_{(A^*)^{\otimes a+b}, (1^{a+b})} (\BGG_{a+b})$, the notation $d^{JT^{(1^{a+b})}}$ denotes the differential of $JT^{(1^{a+b})}$, and $\pi_{a,b} \colon JT^{(1^{a+b})}_{a+b-2} \to A_a^* \otimes A_b^*$ denotes the projection map.  When $a= 0$ or $b= 0$, the corresponding map $\Delta_{a,0}$ or $\Delta_{b,0}$ is simply the identity map on the corresponding graded piece. We define a coproduct
    \[
      \Delta \colon A^* \to A^* \otimes A^*
    \]
    by taking the sum of all the maps $\Delta_{a,b}$. This allows us to define a product
    \[
      m \colon A \otimes A \to A
    \]
    by taking the graded dual of $\Delta$.

    \begin{lemma}
      The complex $JT^{(1^n)}$ is isomorphic to the degree $n$ homogeneous strand of the cobar complex on $A^*$ with coproduct $\Delta$. In particular, $\Delta$ induces a coassociative coalgebra structure on $A^*$ and $m$ gives $A$ the structure of a Koszul algebra.
    \end{lemma}
    
    \begin{proof}
        Proceed by induction on $n$, where the base case $n=2$ is vacuous. For $n > 2$, by Lemma~\ref{lem:JT-tensor}, for any component $A^*_{a_1} \otimes \cdots \otimes A^*_{a_\ell} \subset JT^{(1^n)}_{a_1 + \cdots + a_\ell - \ell}$, the differential has the form\footnote{Notice that in order for the proof to work more smoothly, we use this modified sign convention for the cobar differentials.}
        \[
          d^{JT^{(1^n)}}|_{A^*_{a_1} \otimes \cdots \otimes A_{a^*_\ell}} = \sum_{i=1}^\ell (-1)^{\sum_{j=1}^{i-1} (a_j - 1)} \sum_{\substack{r+s = a_i, \\ r,s > 0}} 1 \otimes \cdots \otimes \underbrace{(-1)^{r+s} \Delta_{r,s}}_{i\text{th spot}} \otimes \cdots \otimes 1.
        \]
        In more detail: if $\ell = 1$ then there is nothing to prove, since this is how the maps $\Delta_{a,b}$ were defined in the first place. For $\ell > 1$, restrict to the subcomplex induced by the inclusion $S_{a_1} \times \cdots \times S_{a_\ell} \subset S_n$. Then the $G$-representation $A_{a_1}^* \otimes \cdots \otimes A_{a_\ell}^*$ is the backmost nonzero representation of the tensor product $JT^{(1^{a_1})} \otimes \cdots \otimes JT^{(1^{a_\ell})}$, and the differential is precisely the differential above, since the term $A^*_{a_i}$ sits in homological degree $a_i-1$ in the complex $JT^{(1^{a_i})}$.

        Thus the complex $JT^{(1^n)}$ is indeed isomorphic to the degree $n$ strand of the cobar complex on $A^*$, and moreover, it follows immediately that this coproduct structure is coassociative, since in order for the cobar differentials to compose to $0$, this is equivalent to coassociativity.

        We deduce that the graded dual $A$ of the coalgebra $A^*$ is an associative algebra. Moreover, the graded duals of the complexes $JT^{(1^n)}$ for each $n$ are isomorphic to the degree $n$ strand of the normalized bar complex for the algebra $A$. The fact that these strands are resolutions is equivalent to the fact that $\tor_i^A (\kk, \kk)_j = 0$ for $i \neq j$. Thus, by definition $A$ is a Koszul algebra.        
    \end{proof}

As a consequence, we deduce a strong realizability result relating $G$-PF sequences and Hilbert functions of Koszul algebras:

    \begin{theorem}\label{thm:GPFalgebraStructure}
    If $A$ is $G$-totally positive, then $([A_i])$ is the equivariant Hilbert function of a Koszul algebra. In particular, every P\'olya frequency sequence is the Hilbert function of a Koszul algebra (in fact, we can take $\kk=\bbq$).
  \end{theorem}

  \begin{proof}
    The first part was proved in the previous lemma. For the second statement, if $(d_i)$ is a P\'olya frequency sequence, then set $A_i = \bbq^{d_i}$. From Proposition~\ref{prop:JTandSchurEquivalence} (taking $G$ trivial), $A$ admits a Jacobi--Trudi structure. Then $(d_i)$ is the Hilbert function of $A$, and from the first part, there is a product on $A$ so that it is a Koszul algebra.
  \end{proof}
  
\begin{remark}
    Although the proof of Theorem \ref{thm:GPFalgebraStructure} requires computing Zelevinsky's functor on the complexes $\BGG_n$ for all $n \geq 2$, the method of computing the Koszul algebra associated to a given $G$-PF sequence is effective once this result is established. This is because one only needs to compute the induced map $A_1 \otimes A_1 \to A_2$, since the kernel of this map will be the defining relations for the induced quadratic algebra structure on $A$. 
\end{remark}

\begin{remark}
  By \cite{piontkovskii2001hilbert} (see also \cite[Ch. 3, Section 5]{polishchuk2005quadratic}), there exist quadratic algebras $A$ and $A'$ with the property that both $A$ and $A'$ \emph{and} $A^!$ and ${A'}^!$ have the same Hilbert series, but such that $A$ is Koszul and $A'$ is not.
  In particular, an associative algebra $A$ whose Hilbert function is a PF sequence need not be Koszul.
\end{remark}

We remark that the equivariant Hilbert functions of Koszul algebras behave much like $G$-PF sequences. For instance:

\begin{prop}[{\cite[Chapter 3]{polishchuk2005quadratic}}, {\cite{backelin2006rates}}, {\cite{froberg1985koszul}}]
  Let $A$ and $B$ be Koszul algebras.
\begin{enumerate}
\item For each $d \ge 1$, the Veronese subalgebras $A^{(d)} = \bigoplus_{n \ge 0} A_{nd}$ are Koszul algebras.
\item The tensor product $A \otimes B$ is a Koszul algebra.
\item The Segre product $\bigoplus_{n \ge 0} A_n \otimes B_n$ is a Koszul algebra.
\end{enumerate}  
\end{prop}
Just as a point of clarification: taking Segre products does \emph{not} necessarily preserve total positivity (see Example \ref{ex:segreNonpositive}) even though it does preserve Koszulness. The following example shows that $([A_i])_{i \geq 0}$ being a $G$-PF sequence is a strictly stronger property than being the equivariant Hilbert function of a Koszul algebra.

\begin{example}
  For this example, we take $G$ to be the trivial group. Consider the sequences
  \[
    a_i = \begin{cases} 1 & \text{if $i=0$}\\ 2 & \text{else} \end{cases},\qquad b_i = (i+1)^2
  \]
  with corresponding generating functions
  \[
    A(t) = \sum_{d \ge 0} a_d t^d = \frac{1+t}{1-t}, \qquad B(t) = \sum_{d \ge 0} b_d t^d = \frac{1+t}{(1-t)^3}.
  \]
  Their Hadamard product has generating function
  \[
    \sum_{d \ge 0} a_d b_d t^d = \frac{4t - 3t^2 + t^3}{(1-t)^3}.
  \]
  Using \cite[Corollary 1.3]{wagner}, $a$ and $b$ are PF-sequences, while their Hadamard product is not.

  On the other hand, $a$ is the Hilbert function of a quadric hypersurface ring in 2 variables, e.g., $\bbc[x_1,x_2]/(x_1^2)$, $b$ is the Hilbert function of a quadric hypersurface ring in 4 variables, e.g., $\bbc[x_1,\dots,x_4]/(x_1^2)$, both of which are Koszul algebras. The Hadamard product is the Hilbert function of the Segre product, which preserves the Koszul property.
\end{example}

\begin{prop}\label{prop:A!schurFunctors}
    If $A$ is a Koszul algebra which is $G$-totally positive, then the quadratic dual algebra $A^!$ is also  $G$-totally positive with
    $$ [\bbs_\lambda^{A^!}] = \det ( [A^!_{\lambda_i^T-\mu_j^T - i + j}])_{i,j=1}^n.$$
\end{prop}

\begin{proof}
    First observe that there is an equality of symmetric functions
    $$e_d = \sum_{\alpha \leq (d)} (-1)^{d-\ell (\alpha)} h_{\alpha_1} \cdots h_{\alpha_{\ell(\alpha)}},$$
    and thus $\psi (e_i) = \sum_{\alpha \leq (d)} (-1)^{d-\ell (\alpha)} [A_{\alpha_1}] \cdots [A_{\alpha_{\ell(\alpha)}}]$; the partial order $\leq$ is the refinement order on compositions. It is well-known (for instance, by taking the Euler characteristic of homogeneous strands of the Bar complex) that this latter sum is precisely the K-class of $[(A^!_d)^*]$.
\end{proof}
As an immediate consequence, we obtain:

\begin{theorem}\label{thm:koszulnessAndJTStructure}
    Let $A$ be a Koszul algebra admitting a Jacobi--Trudi structure. Then $A^!$ also admits a Jacobi--Trudi structure, obtained by applying the transpose functor to the Jacobi--Trudi structure on $A$, then taking the graded dual.
  \end{theorem}
  
\begin{example}
    This duality can be seen in both Examples \ref{ex:symJTstructure} and \ref{ex:tensorJTstructure}. Applying the transpose functor to $S^\bullet (V \otimes -)$ yields the functor $\bigwedge^\bullet (V \otimes -)$, and dualizing yields the functor $\bigwedge^\bullet (V^* \otimes -)$. For similar reasons as the symmetric algebra case, there is indeed an isomorphism
    $$\bigwedge^\bullet (V^* \otimes \bbc^n) \cong (\bigwedge^\bullet V^*)^{\otimes n}$$
    satisfying $(1) - (3)$ of Definition \ref{def:JTAlgebras}. Likewise, applying the transpose functor, dualizing, and plugging in $\bbc^n$ for the Segre product $T^\bullet (V) \circ S^\bullet (\bbc^n)$ yields the Segre product $T^\bullet (V^*) \circ \bigwedge^\bullet \bbc^n$. Note that this algebra only has nonzero weight spaces for squarefree weights (i.e., with all entries $\leq 1$), and indeed for a squarefree weight $\lambda$ there is an isomorphism
    \[
      (T^\bullet (V^*) \otimes \bigwedge^\bullet (\bbc^n))_\lambda \cong T^{|\lambda|} (V).
      \]
    This makes sense since the quadratic dual of the tensor algebra is the trivial extension $T^\bullet (V^*) / (V^* \otimes V^*) = \bbc \oplus V^*$, which is $0$ in degrees $\geq 2$. 
\end{example}

\begin{prop}
    Assume $A$ is a $G$-totally positive Koszul algebra. If $\lambda / \mu$ is a ribbon diagram, then $s_{\lambda / \mu}^A$ is the class of the ribbon Schur functors as defined in \cite{vandebogert2023ribbon}.
\end{prop}

\begin{proof}
  Since both definitions of the ribbon Schur functors for a Koszul algebra satisfy the same Jacobi--Trudi identities \cite[Theorem A.12]{vandebogert2023ribbon}, the result follows.
\end{proof}

\begin{remark}
    Ribbon Schur functors can be defined for any algebra since they only require the existence of an associative multiplication. One of the main results of \cite{vandebogert2023ribbon} is that these ribbon Schur functors satisfy the Jacobi--Trudi identity for all compositions \emph{if and only if} the underlying algebra $A$ is Koszul. Thus every Koszul algebra is partially $G$-totally positive. 
\end{remark}

\begin{remark}\label{rk:partialPositivity}
    Continuing with the theme of ``partial" $G$-positivity, we can imagine a spectrum of algebras with arbitrary graded algebras on one end and $G$-totally positive algebras on the other, parallel to the ${\rm PF}_\ell$ property. More precisely, a sequence is ${\rm PF}_\ell$ if all minors of order $\leq \ell$ of the associated Toeplitz matrix are non-negative. On the algebraic end, this translates into the property that $A^{\otimes n}$ admits a $\GL_n$-action for all $n \leq \ell$, and, employing Zelevinsky's functor as in the proof of Theorem \ref{thm:GPFalgebraStructure}, we see that the $G$-${\rm PF}_3$ property is enough to guarantee that a sequence $([A_i])_{i \geq 0}$ is the Hilbert function of some associative algebra which is generated in degree 1. (A sequence that satisfies the $G$-${\rm PF}_2$-property, i.e., log-concavity, can be realized as the Hilbert function of an algebra, not necessarily associative, which is generated in degree 1.)
\end{remark}

\section{Quadric Schur functors}\label{sec:quadricSchur}

In Section~\ref{sec:JTstructures}, we defined the notions of $G$-total positivity and Jacobi--Trudi structures and established their equivalence. Thus far the only examples of Jacobi--Trudi structures that we have seen come from building blocks induced by the (essentially trivial) Jacobi--Trudi structure on the symmetric algebra. In this section, we construct a significantly more subtle Jacobi--Trudi structure on quadric hypersurface rings. In the next section, we then use the induced Schur functors to construct explicit pure free resolutions over the quadric hypersurface ring.

In this section, we will assume that our underlying field has characteristic 0; for emphasis, we will just assume that it is the complex numbers $\bbc$.

\subsection{Jacobi--Trudi structures on quadric hypersurface rings}

Let $V$ be a vector space and let $q \in S^2(V^*)$ be a nonzero quadratic form on $V$. Let $\rad q$ be the radical of $q$, and choose a complement $V_\circ$ of $\rad q$ so that we have a vector space decomposition
\[
  V = \rad q \oplus V_\circ.
\]
We will call the triple $(V,q,V_\circ)$ as above a \defi{framed orthogonal space}. We define ${\rm O}(V,q,V_\circ)$ to be the subgroup of $\GL(V)$ that preserves $q$ and the above direct sum decomposition. 

In particular, $q$ gives a canonical isomorphism $V_\circ \cong V_\circ^*$, so we also have a dual quadric $q \in S^2(V_\circ)$ by taking the image of $q$ under this isomorphism. We will identify $q$ with its image under the inclusion $S^2(V_\circ) \subseteq S^2(V)$.
  
\begin{construction}\label{cons:UGVECons}
  Let $(V,q,V_\circ)$ be a framed orthogonal space and let $E$ be a vector space. Define the $\bbz$-graded Lie algebra $\g_{V,E}$ with components
    \[
      (\g_{V,E})_1 = V \otimes  E, \qquad (\g_{V,E})_2= \bigwedge^2 E,
    \]
    equipped with the Lie bracket (for $x\otimes e, x' \otimes e' \in V \otimes E$)
    \[
      [x \otimes e , x' \otimes e'] \coloneq q(x,x') e \w e'.
    \]
    All other Lie brackets are zero by degree reasons. There is an inclusion $S^2 (E \otimes V) \subset \uU_2 (\g_{V,E})$ obtained by identifying (via PBW) $\uU_2 (\g_{V,E}) \cong S^2 (E \otimes V) \oplus \bigwedge^2 E$. Using the inclusion $q \otimes  S^2 (E) \subset S^2 (V \otimes  E) \subset U_2 (\g_{V,E})$, let $I_{V,E}$ denote the left $\uU (\g_{V,E})$-ideal generated by $q \otimes  S^2 (E)$ inside of $\uU (\g_{V,E})$. 

    Similarly, we define the $\bbz$-graded Lie superalgebra $\g^!_{V,E}$ with components
    \[
      (\g^!_{V,E})_1 = V^* \otimes  E, \qquad (\g^!_{V,E})_2 = S^2 (E),
    \]
    equipped with the bracket
    \[
      [x \otimes e , x' \otimes e'] \coloneq q(x,x') e \cdot e'.
    \]
    (Collapsing the $\bbz$-grading to a $\bbz/2$-grading gives its superalgebra structure.)
    Notice that there is similarly an inclusion $q \otimes  \bigwedge^2 E \subset \bigwedge^2 (V^* \otimes  E) \subset U_2 (\g^!_{V,E})$. Let $I^!_{V,E}$ denote the left $\uU (\g^!_{V,E})$-ideal generated by $q \otimes  \bigwedge^2 E$ inside of $\uU (\g^!_{V,E})$.
  \end{construction}

\begin{example}\label{ex:rankOneComputation}
    When $\dim V = 1$, the quotient $\uU (\g_{\bbc , E}) / I_{\bbc , E}$ may be identified with a more familiar representation. Choose a basis $e_1 , \dots , e_n$ for $E$; by definition, the enveloping algebra $\uU ( \g_{\bbc , E})$ has relations of the form $e_i \otimes e_j - e_j \otimes e_i = e_i \w e_j$ for all $1 \leq i,j \leq n$. Taking the quotient by $S^2 (E)$ further imposes the relations $e_i \otimes e_j + e_j \otimes e_i = 0$, and thus 
    $$e_i \otimes e_j = \frac{1}{2} \cdot  e_i \w e_j \quad \text{for all} \ 1 \leq i,j \leq n.$$
    This means that as a $\gl (E)$-representation, the quotient $\uU ( \g_{\bbc , E}) / I_{\bbc , E}$ is precisely the spinor representation of $\so (E \oplus E^*)$, up to $1/2$-times the trace representation on $\gl (E)$.  Furthermore, this shows that $I_{\bbc, E}$ is a two-sided ideal and $\uU (\g_{\bbc, E})/I_{\bbc, E}$ is isomorphic to the exterior algebra on $E$.
\end{example}

\begin{theorem}[{\cite{charCoincidence}}]
  Assume $V$ is a framed orthogonal space equipped with a quadratic form $q$ of odd rank and set $A \coloneqq S^\bullet (V) / (q)$. Then we have an isomorphism of $\GL(E)$-representations
  \[
    Z_{V,E} \coloneqq \uU (\g_{V,E}) / I_{V,E} \cong A^{\otimes \dim E}
  \]
  which induces a Jacobi--Trudi structure on $A$.
\end{theorem}

\begin{example} \label{ex:rank1-super}
  If $q$ has rank 1, we can assign a $\bbz/2$-grading to $V$ by $V_0 = \rad q$ and $V_1 = V_\circ$ and we take $G = \GL(V_0|V_1)$, the general linear supergroup. Then the symmetric algebra $S^\bullet(V)$ (treating $V$ as a superspace) is isomorphic to $A$ (in the usual definition) as a $G$-representation. Furthermore, it follows that we have an isomorphism of algebras
  \[
    \uU(\g_{V,E}) \cong S^\bullet(V_0 \otimes E) \otimes \uU(\g_{V_1,E}).
  \]
  From this, we see that the ideal generated by $S^2(E)$ is a 2-sided ideal and, using Example~\ref{ex:rankOneComputation}, that the quotient algebra has an explicit description as
  \[
    \uU(\g_{V,E})/(S^2 E) \cong S^\bullet(V_0 \otimes E) \otimes \bigwedge(V_1 \otimes E) \cong S^\bullet(V \otimes E),
  \]
  where the last symmetric algebra is treating $V$ as a $\bbz/2$-graded vector space.

  Since $A$ is a polynomial representation, $G$ acts semisimply on all tensor powers. In fact, its Jacobi--Trudi structure is a special case of the standard Jacobi--Trudi structure on the symmetric algebra.
\end{example}

In the following, we assume as above that $V$ is a framed orthogonal space equipped with a quadratic form $q$ of odd rank and $A$ is the associated quadric hypersurface ring. 

\begin{definition}[Quadric Schur functors]\label{def:quadricSchurFunctor}
  Let $\mu \subseteq \lambda$ be partitions with $n \ge \ell(\lambda)$. The \defi{quadric Schur functor} is defined to be
  \[
    \bbs_{\lambda / \mu}^A = \hom_{\GL_n}(\bbs_{\lambda/\mu}(\bbc^n), {\rm U}(\g_{V,\bbc^n})/I_{V,\bbc^n}).
  \]
  Likewise, the \defi{dual quadric Schur functor} is defined to be
  \[
    \bbs_{\lambda / \mu}^{A^!} = \hom_{\GL_n}(\bbs_{\lambda/\mu}(\bbc^n), {\rm U}(\g^!_{V,\bbc^n})/I^!_{V,\bbc^n}). \qedhere
  \]
\end{definition}

\begin{remark}\label{rk:submaximalPositivity}
    When $A$ is a quadric hypersurface ring associated to an \emph{even} rank quadric, the virtual character
    $$[\bbs^A_\lambda ] = \det ( [A_{\lambda_i - i + j} ] )_{i,j=1}^{\ell (\lambda)}$$
    is not necessarily the character of any $\SO (2n)$-representation when $\ell (\lambda ) > \dim V$. For instance, Example~\ref{ex:segreNonpositive} shows that when $n=4$, the virtual character $s_{2,2,2}^A$ is not an actual character (see Remark~\ref{rk:positivityFailure} for more details).

    To fix this, we must restrict to a submaximal torus (in other words, set $x_n = 1$, where $\dim V = 2n$). This allows us to view the virtual character $[\bbs^A_\lambda ]$ as living in the character ring ${\rm K} (\SO (2n-1))$, in which case there is a familiar looking description in terms of the odd-orthogonal quadric Schur characters:
    $$[\bbs^A_\lambda ] = \sum_{\substack{\mu \subset \lambda \\ \lambda / \mu  \ \text{horizontal strip}}} [\bbs^{A'}_\mu ],$$
    where $A'$ denotes the odd rank quadric hypersurface ring associated to $\SO (2n-1)$. This follows from the isomorphism of character rings taken advantage of in the work \cite{charCoincidence}: the ``even" quadric Schur functor should correspond to the odd-symplectic character \cite{proctor1988odd} of the classically defined Schur functor $\bbs_\lambda (\bbc^{2n-1|1})$; the classical $\GL_{n-1} \hookrightarrow \GL_n$ branching rule yields the equality
    $$\bbs_\lambda (\bbc^{2n-1|1}) = \bigoplus_{\substack{\mu \subset \lambda \\ \lambda / \mu  \ \text{horizontal strip}}} \bbs_{\mu} (\bbc^{2n-2|1}).$$
    Since the $\Spo(2n-2|1)$-character of $\bbs_{\mu} (\bbc^{2n-2|1})$ is precisely the ${\rm O} (2n-1)$-character of the odd quadric Schur module $\bbs^A_\mu$, the character equality posed above follows (and thus we obtain positivity of the type D quadric hypersurface ring after restriction to a submaximal torus). 
\end{remark}

Because of this subtlety between the positivity properties of the type B and D cases, whenever we refer to a quadric Schur module without any reference to the ambient quadric hypersurface, we can only view these objects as representations of an odd-orthogonal group in general.

\begin{prop} \label{prop:quadric-schurdim}
Set $m=\dim V$. Then $\bbs_{\lambda}^A=0$ if and only if $\lambda_m>1$. In particular, $A$ has Schur dimension $(m-1)|1$.
\end{prop}

\begin{proof}
  From Proposition~\ref{prop:JT-det}, $\dim \bbs^{\lambda}_A$ is independent of the quadric $q$ since $\dim A_d$ is. In particular, it suffices to prove this statement when $q$ has rank 1 and $q = x_1^2$. By Example~\ref{ex:rank1-super}, $\bbs_\lambda^A$ is the usual Schur functor $\bbs_\lambda(\bbc^{m-1|1})$, so we can appeal to Example~\ref{ex:schur-dim}.
\end{proof}

\begin{example}
  Suppose that $q$ has full rank. 
    The quadric Schur functors have more explicit descriptions in terms of orthogonal group representations in some special cases. Firstly, when $\lambda = (d)$ the quadric Schur functor $\bbs_{(d)}^A$ is equal to $A_d$, which in turn is the irreducible ${\rm O} (V)$-representation $S^d(V)/S^{d-2}(V)$ of highest weight $d$. When $\lambda = (1^d)$, the quadric Schur functor $\bbs_{(1^d)}^A$ is precisely the (dual of the) degree $d$ piece of the quadratic dual $(A^!_d)^*$, which admits a direct sum decomposition
    \[
      \bbs_{(1^d)}^A = \bigwedge^d V \oplus \bigwedge^{d-2} V \oplus \bigwedge^{d-4} V \oplus \cdots . \qedhere
    \]
  \end{example}

  For general $\lambda$, we can give a uniform direct sum decomposition in the ``stable range''. In what follows, given a partition $\mu$ with $2\ell(\mu) \le \dim V$, we use $\bbs_{[\mu]}(V)$ to denote the irreducible representation of ${\rm O}(V)$ with highest weight $\mu$.

  \begin{prop} \label{prop:stable-decomp}
    Suppose that $q$ has full rank and assume that $2\ell (\lambda) \leq \dim V$. There is an ${\rm O}(V)$-equivariant isomorphism:
    \[
      \bbs_\lambda^A \cong \bigoplus_{\mu , \nu} \bbs_{[\mu]} (V)^{\oplus c^{\lambda}_{\mu, (2 \nu)^T}},
    \]
    where in the above, $c^{\lambda}_{\mu, (2 \nu)^T}$ denotes the Littlewood--Richardson coefficient. 
  \end{prop}

  \begin{proof}
    We will use the theory of the universal character ring of the orthogonal group from \cite[\S 2.1]{koiketerada}. This gives a basis $\chi_O(\lambda)$ of the ring of symmetric functions $\Lambda$ indexed by partitions $\lambda$ such that $s_\lambda = \sum_{\mu,\nu} c^{\lambda}_{\mu, 2\nu} \chi_O(\mu)$ \cite[Theorem 2.3.1]{koiketerada}. In particular, we have $s_{1,\dots,1} = \chi_O(1,\dots,1)$. It follows from the dual Jacobi--Trudi identity that for any $r \ge \lambda_1$, we have
    \[
      \sum_{\mu,\nu} c^{\lambda^T}_{\mu, 2\nu} \chi_O(\mu) = s_{\lambda^T} = \det(s_{(1^{\lambda_i - i + j})})_{i,j=1}^r = \det(\chi_O(1^{\lambda_i - i + j}))_{i,j=1}^r
    \]

    Also, there is a ring involution $i_O$ on $\Lambda$ such that $i_O(\chi_O(\lambda)) = \chi_O(\lambda^T))$ \cite[Theorem 2.3.4]{koiketerada}. Applying this to the identity above gives
    \[
            \det(\chi_O(\lambda_i - i + j))_{i,j}^{r} = \sum_{\mu,\nu} c^{\lambda^T}_{\mu, 2\nu} \chi_O(\mu^T) = \sum_{\mu,\nu} c^{\lambda}_{\mu, (2\nu)^T} \chi_O(\mu)
          \]
          where in the last sum, we reindexed and used that the Littlewood--Richardson coefficients are invariant upon applying the transpose involution. Finally, we use the specialization map $\pi \colon \Lambda \to {\rm K}({\rm O}(V))$ \cite[\S 2.4]{koiketerada}. We just need to know that this is a ring homomorphism and that if $2\ell(\lambda) \le \dim V$, then $\pi(\chi_O(\lambda))$ is the character of the irreducible representation $\bbs_{[\lambda]}(V)$. Applying this to the last identity finishes the proof.
        \end{proof}

        \begin{remark}
          The proof above actually gives the character of $\bbs_\lambda^A$ for an odd quadric hypersurface in general, but the answer depends on computing the terms $\pi(\chi_O(\lambda))$. In general, this is either 0 or the character of an irreducible representation up to a sign and is computed by a combinatorial rule. Since the end result is the character of a representation, the negative signs must cancel with positive signs, but this method does not make clear which representations appear with positive coefficient in the case that $2\ell(\lambda) > \dim V$. Similarly, in the type D case, we can only guarantee that the coefficients will be non-negative after restricting to a submaximal torus. 
        \end{remark}

\begin{remark}\label{rmk:sympletic}
  In \cite[Theorem 4.3.4]{stability-patterns}, it is shown that there is a duality between the algebraic representations of the infinite rank symplectic group ${\rm Sp}(\infty)$ and the infinite rank orthogonal group ${\rm O}(\infty)$. This duality predicts that we can define a symplectic version of the above constructions.

  In this case, we should replace the quadric hypersurface with the enveloping algebra of the Heisenberg Lie algebra $U \oplus \bbc$ where $U$ is our symplectic space. Then we are instead looking at an algebra whose $d$th graded piece is isomorphic to $\Sym^d U \oplus \Sym^{d-2} U \oplus \cdots$. However, if $\dim U = 2$, then the Jacobi--Trudi determinant for the partition $(1,1,1)$ gives something with dimension $\det \begin{bmatrix} 2 & 4 & 6 \\ 1 & 2 & 4 \\ 0 & 1 & 2 \end{bmatrix} = -2$, so there cannot be a Jacobi--Trudi structure on such an algebra. On the other hand, the stable representation theory of the symplectic group combined with the results of the next section instead imply that the enveloping algebra of the Heisenberg Lie algebra admits a partial Jacobi--Trudi structure as in Remark \ref{rk:partialPositivity}. More precisely, $A^{\otimes n}$ admits a $\GL_n$-action with the desired weight spaces for $n \leq \dim U / 2$. 
\end{remark}

\subsection{Type D quadric Schur modules: the stable case} \label{ss:geom}

 In this section, we prove that even though the even rank quadric hypersurface does not admit a full (equivariant) Jacobi--Trudi structure in general, it does admit a \emph{partial} Jacobi--Trudi structure as in Remark \ref{rk:partialPositivity}. This result will be crucial for proving that the functors $\Phi^X$ of the next subsection are suitably well-behaved.

For this section, we will assume that $q$ has full rank and that $\dim V = 2d$. In that case, we must have $V=V_\circ$ in the definition of framed orthogonal space, so we will suppress this from the notation.

A subspace $W \subset V$ is isotropic if $q(v,w) = 0$ for all $v,w \in W$. The isotropic Grassmannian $\igr(e,V)$ is the variety which parametrizes all $e$-dimensional isotropic subspaces of $V$. If $d>e$, then this is a connected variety. Otherwise, $\igr(d,V)$ has 2 connected components which are isomorphic to one another. We can distinguish the two by fixing a reference isotropic subspace $R_0$ of dimension $d$; the function $\igr(d,V) \to \bbz/2$, given by $R \mapsto \dim(R \cap R_0) \pmod 2$, distinguishes the components. We will choose one (the choice is irrelevant for our purposes) and denote it by $\igr^+(d,V)$. The tautological isotropic subbundle $\R$ is the restriction of the standard rank $d$ tautological bundle on the Grassmannian to $\igr^+ (d,V)$. 

    Throughout this section, we will set
    \[
      X = \igr^+(d,V).
    \]
    Finally, we note that the quadratic form induces an isomorphism $V/\R \cong \R^*$.

    Below, we let $\underline{\g_{V,E}} = \g_{V,E} \times X$, which we think of as a trivial sheaf of Lie algebras over $X$.
    
    \begin{prop} \label{thm:largeRankTorIso}  
  Let $V$ be a $2d$-dimensional orthogonal space with $d > \dim E$ equipped with a full rank quadric $q$.  Then:
    \begin{enumerate}
        \item The bundle $\R \otimes E$ is a Lie subalgebra subsheaf in $\underline{\g_{V,E}}$.

        \item Define the $\uU (\underline{\g_{V,E}})$-module
          \[
            \eta \coloneq \uU( \underline{\g_{V,E}})  \otimes_{\uU ( \R \otimes E)} \OO_X.
          \]
          There is an exact complex of vector bundles
          \[
            0 \to \bigwedge^{d \cdot \dim E} (\R \otimes  E) \otimes  \uU (\underline{\g_{V,E}}) \to \cdots \to (\R \otimes  E) \otimes  \uU (\underline{\g_{V,E}}) \to \uU (\underline{\g_{V,E}}) \to \eta \to 0,
          \]
          and in particular, ${\rm H}^0(X, \eta)$ is a left $\uU(\g_{V,E})$-module.
          
        \item There is an isomorphism 
    $$\tor_i^{\uU (\g_{V,E})} (\bbc , {\rm H}^0 (X , \eta) )  \cong \bigwedge^i S^2 (E),$$
    and an isomorphism of $\uU (\g_{V,E})$-modules ${\rm H}^0 (X, \eta) \cong \uU (\g_{V,E}) / I_{V,E}$. 
    \end{enumerate}
\end{prop}

\begin{proof}
    \textbf{Proof of (1):} This follows by definition of the Lie bracket on $\g_{V,E}$.

    \textbf{Proof of (2):} Apply $\uU (\underline{\g_{V,E}}) \otimes_{\uU (\R \otimes E)} -$ to the Chevalley--Eilenberg resolution of $\uU (\R \otimes E)$ (this preserves exactness since the extension $\uU (\R \otimes E) \subset \uU (\underline{\g_{V,E}})$ is flat).

    \textbf{Proof of (3):}
    For the moment, ignore the direct summand $\bigwedge^2 E$ so that we have a short exact sequence
    \[
      0 \to \R \otimes E \to (V \otimes E) \otimes \OO_X \to \R^* \otimes E \to 0.
    \]
    Identifying the middle term with the functions on $\hom(E,V^*) \times X$, we have a projection map ${\rm Spec}_X(S^\bullet(\R^* \otimes E)) \to \hom(E,V^*)$. The image consists of linear maps $E \to V^*$ whose image is contained in a subspace of the form $W^*$ for some isotropic subspace $W$ that is parametrized by the component $X$ in $\igr(d,V)$; since $2\dim W = \dim V$, this implies that $W^*$ is isotropic as well. Since $\dim E < d$, this simply means that the image of $E$ is an isotropic subspace. In particular, in the notation of \cite[\S 3.3]{sam2015littlewood}, the image of this projection is the variety $Y_{\dim E, \dim E, d}$. By \cite[Theorem 3.8]{sam2015littlewood}, this variety has rational singularities and is a complete intersection with ideal generated by $S^2E$.
    
    Define the bundle of algebras $R \coloneq S^\bullet (\R^* \otimes  E) \otimes  S^\bullet \left( \bigwedge^2 E \right)$. Then \cite[\S 5.1]{weyman2003} and the above facts imply that there is an isomorphism
    $$\tor_i^{S^\bullet (V^* \otimes  E) \otimes  S^\bullet \left( \bigwedge^2 E \right)} (\bbc , {\rm H}^0 (X , R)) = \bigoplus_{j \geq 0} {\rm H}^j \left( X , \bigwedge^{i+j} (\R \otimes  E) \right).$$
Similarly, we also get an isomorphism (the rational singularities claim is equivalent to the higher cohomology of $\eta$ vanishing, but by PBW this has the same underlying vector bundle as $R$)
    $$\tor_i^{\uU (\g_{V,E})} \left( \bbc , {\rm H}^0 (X, \eta) \right) =  \bigoplus_{j \geq 0} {\rm H}^j \left( X , \bigwedge^{i+j} (\R \otimes  E) \right).$$
    Since the righthand side of both of the above equalities are evidently equal, there are isomorphisms
    $$\tor_i^{\uU (\g_{V,E})} \left( \bbc , {\rm H}^0 (X , \eta) \right) \cong \tor_i^{S^\bullet (V^* \otimes  E) \otimes  S^\bullet \left( \bigwedge^2 E \right)} (\bbc , {\rm H}^0 (X , R)) \cong \bigwedge^i S^2 (E).$$
    
    The final isomorphism ${\rm H}^0 (X , \eta) \cong \uU (\g_{V,E}) / I_{V,E}$ is an immediate consequence of the above Tor computation, since ${\rm H}^0 (X , \eta)$ is thus obtained as the quotient of $\uU (\g_{V,E})$ by the left ideal generated by $S^2 (E)$; since there is a unique such subrepresentation, the result follows.    
  \end{proof}

\begin{remark} \label{rmk:comm-degen}
  A consequence of the previous proof is that, when $\dim V = 2d$ and $d>\dim E$, $S^2 E$ is a Gr\"obner basis under the PBW degeneration of $\uU(\g_{V,E})$ to the commutative algebra $S^\bullet(\g_{V,E})$ and the resulting ideal is a complete intersection.

  This can fail if the assumptions are not met. For instance, take $\dim E = \dim V = 2$, so that after choosing bases this ideal has 3 generators
  $$(x_{11}^2 + x_{21}^2 , x_{11} x_{12} + x_{21} x_{22} , x_{12}^2 + x_{22}^2) \subset \bbc [x_{11} , x_{12} , x_{21} , x_{22}  ],$$
  and has codimension 2.
\end{remark}

\begin{cor}
    Let $V$ be an orthogonal space of dimension $2n$ and let $A = S^\bullet (V) / (q)$ be the associated quadric hypersurface ring. Then the tensor powers $A^{\otimes k}$ admit a $\GL_k$-action satisfying $(1)-(3)$ of Definition \ref{def:JTAlgebras} for all $k \leq n$. Equivalently, for all partitions $\lambda$ with $\ell (\lambda) \leq n$, there exists an ${\rm O} (V)$-representation $\bbs^A_\lambda$ satisfying the Jacobi--Trudi identity with respect to $A$. 
\end{cor}

\begin{proof}
    We prove the more precise statement that for every weight $\mu \in \bbz^k$ with $k \leq n$, there is an equality
    $$(\uU (\g_{V,E}) / I_{V,E})_\mu \cong A_{\mu_1} \otimes  A_{\mu_2} \otimes  \cdots \otimes  A_{\mu_{\dim E}},$$
    where $k = \dim E$. We proceed by induction on $\dim E$. When $\dim E = 1$, there is an isomorphism of $\bbc$-vector spaces:
    \[
      \uU (\g_{V,E} ) \cong S^\bullet (V \otimes E).
    \]
    The quotient by $S^2(E)$ translates to the quotient by the principal ideal generated by the defining quadric $q$. Thus $\uU (\g_{V,E}) / I_{V,E} \cong S^\bullet (V) / (q)$, and restricting to weight spaces corresponds to restricting to graded pieces.

    Assume now that $\dim E =: k > 1$ and write $E = E' \oplus \bbc e$. There is a direct sum decomposition $\bigwedge^2 E \cong \bigwedge^2 E' \oplus E' \otimes e$, and combining this with Proposition \ref{thm:largeRankTorIso}(3) yields an equality of multigraded Hilbert series:
    \[
      \hs (\uU (\g_{V,E})/I_{V,E}) = \frac{  \displaystyle\prod_{1 \leq i \leq j \leq k-1} (1 - x_i x_j) \cdot \prod_{1 \leq i \leq k-1} (1-x_i x_k) \cdot (1-x_k^2) }{\displaystyle \prod_{1 \leq i \leq k-1} (1 - [V] x_i) \cdot (1- [V] x_k) \cdot \prod_{1 \leq i < j \leq k-1} (1 - x_i x_j) \cdot \prod_{1 \leq i \leq k-1} (1 - x_i x_k)},
    \]
     where the above Hilbert series is viewed as having coefficients in the ring ${\rm K}(G)[x_1 , \dots , x_k]$.  The term $\prod_{1 \leq i \leq k-1} (1 - x_i x_k)$ cancels in the above expression, whence there is an equality of multigraded Hilbert series
     \[
       \hs (\uU (\g_{V,E}) / I_{V,E} ) = \hs ( \uU (\g_{V,E'} ) / I_{V , E'}) \cdot \hs (\uU (\g_{V,\bbc e} ) / I_{V , \bbc e}).
       \]
    Computing the component of $\uU (\g_{V,E}) / I_{V,E}$ with weight $\mu = ( \mu_1 , \dots , \mu_k)$ is equivalent to computing the K-class coefficient of $x_1^{\mu_1} \cdots x_n^{\mu_k}$. By the above equality combined with the inductive hypothesis, this K-class is precisely $[A_{\mu_1} \otimes  A_{\mu_2} \otimes  \cdots \otimes  A_{\mu_k} ]$.
\end{proof}

\begin{remark}\label{rk:positivityFailure}
    As alluded to previously, the even-orthogonal hypersurface ring does not necessarily admit a Jacobi--Trudi structure. An easy example of this is already provided by Example \ref{ex:segreNonpositive}: assume that $\dim V = \dim W = 2$; the Segre product $S^\bullet (V) \circ S^\bullet (W)$ may be realized as a quadric hypersurface ring corresponding to some even rank quadric $q$ (concretely, this is just the quotient $\bbc [x_{11}, x_{12}, x_{21} , x_{22}]/(x_{11} x_{22} - x_{12} x_{12} )$). The isomorphism $\so (4) \cong \sl (2) \times \sl (2)$ implies that non-positivity of the virtual character $[\bbs^{S^\bullet (V) \circ S^\bullet (W)}_{(2,2,2)}]$ contradicts the existence of a Jacobi--Trudi structure on the even rank quadric hypersurface ring. A direct computation shows that the restriction to the diagonal $ \so (3) \cong \sl (2) \subset \sl (2) \times \sl (2) \cong \so(4)$ \emph{is} positive, as expected by Remark \ref{rk:submaximalPositivity}. 
\end{remark}

\subsection{The functors $\Phi^X$} \label{ss:phi}

In this section, we use the Jacobi--Trudi structure on quadric hypersurface rings to define functors $\Phi^X$  on the category of $\GL_n$-representations. The functors defined here will not generally be monoidal, but we will prove a weaker statement about the behavior of these functors on certain types of representations that will end up being enough to construct general Pieri-type resolutions over any quadric hypersurface. 

Let $R$ be a $\bbc$-algebra. We say that $X$ is a \defi{polynomial $R \times \GL(E)$-module} if it has commuting actions of $R$ and $\GL(E)$ which makes $X$ a polynomial representation of $\GL(E)$; in addition, we say that $X$ is a \defi{polynomial $R \times \GL(E)$-algebra} if it is a polynomial $R \times \GL(E)$-module equipped with an associative algebra structure whose product is $R \times \GL(E)$-equivariant. Finally, we define an \defi{$X$-module} to be a polynomial $R \times \GL(E)$-module $Y$ such that the product map $X \otimes Y \to Y$ is $R \times \GL(E)$-equivariant.

\begin{definition}\label{def:thePhiFunctors}
    Let $E$ be a vector space and $X$ a polynomial $R \times \GL(E)$-module. Define the functor
    \[
      \Phi^X \colon \operatorname{Rep} (\GL (E)) \to R\operatorname{-Mod}, \qquad
      \Phi^X (M) \coloneq (X \otimes M)^{\GL (E)},
    \]
    where $(-)^{\GL(E)}$ denotes $\GL (E)$-invariants. 
\end{definition}

\begin{prop}\label{prop:monoidalPhiFunctors}
  Let $X$ be a polynomial $R \times \GL(E)$-algebra and let $Y$ be an $X$-module.
  \begin{enumerate}
        \item There is a canonical morphism of bifunctors
          \[
            \Phi^X (-)  \otimes \Phi^Y (-)  \to \Phi^Y (- \otimes -).
          \]
        \item If $M$ has a $\GL (E)$-equivariant algebra structure, then $\Phi^X (M)$ is an algebra with product
          \[
            \Phi^X(M) \otimes \Phi^X(M) \to \Phi^X(M \otimes M) \to \Phi^X(M)
          \]
          where the first map comes from (1) and the second map is $\Phi^X$ applied to the product map on $M$.
        \item If $N$ is any $\GL (E)$-representation with a $\GL(E)$-equivariant module structure over $M$ (where $M$ is as in $(2)$), then $\Phi^Y (N)$ is a module over $\Phi^X (M)$. 
    \end{enumerate}
\end{prop}

\begin{proof}
    \textbf{Proof of (1):} Let $M$ and $N$ be any $\GL (E)$-representations. Then there is a canonical string of maps:
    \begingroup\allowdisplaybreaks
    \begin{align*}
        \Phi^X (M) \otimes \Phi^Y (N) =& (X \otimes M)^{\GL (E)} \otimes (Y \otimes N)^{\GL(E)} \\
        \to& (X \otimes M \otimes Y \otimes N)^{\GL(E)} \\
        \cong& ( (X \otimes Y) \otimes (M \otimes N) )^{\GL(E)} \\
        \to& (Y \otimes (M \otimes N))^{\GL(E)} = \Phi^Y (M \otimes N),
    \end{align*}
    \endgroup
    where the final map $( (X \otimes Y) \otimes (M \otimes N) )^{\GL(E)} \to (Y \otimes (M \otimes N))^{\GL(E)}$ is induced by taking $\GL$-invariants of the module action $(X \otimes Y \to Y) \otimes (M \otimes N)$.  
    
    \textbf{Proof of (2):} Set $Y \coloneq X$ and $N \coloneq M$ in the canonical map of $(1)$, then compose with the map obtained by taking $\GL$-invariants of the multiplication $M \otimes M \to M$. Notice that this yields an associative algebra structure by the following commutative diagram:
    % https://q.uiver.app/#q=WzAsNCxbMCwwLCIoWCBcXG90aW1lcyBNKV57XFxHTH0gXFxvdGltZXMgKFggXFxvdGltZXMgTSlee1xcR0x9IFxcb3RpbWVzIChYIFxcb3RpbWVzIE0pXntcXEdMfSJdLFsyLDAsIlxcbGVmdCggKFggXFxvdGltZXMgWCkgXFxvdGltZXMgKE0gXFxvdGltZXMgTSkgXFxyaWdodClee1xcR0x9IFxcb3RpbWVzIChYIFxcb3RpbWVzIE0pXntcXEdMfSJdLFswLDEsIihYIFxcb3RpbWVzIE0pXntcXEdMfSBcXG90aW1lcyAoKFggXFxvdGltZXMgWCkgXFxvdGltZXMgKE0gXFxvdGltZXMgTSkpXntcXEdMfSAiXSxbMiwxLCIoKFggXFxvdGltZXMgWCBcXG90aW1lcyBYKSBcXG90aW1lcyAoTSBcXG90aW1lcyBNIFxcb3RpbWVzIE0pKV57XFxHTH0iXSxbMCwxXSxbMCwyXSxbMiwzXSxbMSwzXV0=
\[\begin{tikzcd}
	{(X \otimes M)^{\GL} \otimes (X \otimes M)^{\GL} \otimes (X \otimes M)^{\GL}} && {\left( (X \otimes X) \otimes (M \otimes M) \right)^{\GL} \otimes (X \otimes M)^{\GL}} \\
	{(X \otimes M)^{\GL} \otimes ((X \otimes X) \otimes (M \otimes M))^{\GL} } && {((X \otimes X \otimes X) \otimes (M \otimes M \otimes M))^{\GL}}
	\arrow[from=1-1, to=1-3]
	\arrow[from=1-1, to=2-1]
	\arrow[from=1-3, to=2-3]
	\arrow[from=2-1, to=2-3]
\end{tikzcd}\]

    \textbf{Proof of (3):} Compose the canonical map $\Phi^X (M) \otimes \Phi^Y (N) \to \Phi^Y (M \otimes N)$ of $(1)$ with the map $M \otimes N \to N$. This map will be compatible with the algebra structure of $\Phi^X (M)$ by the commutativity of the following diagram (analogous to that used for the proof of $(2)$):
    % https://q.uiver.app/#q=WzAsNCxbMCwwLCIoWCBcXG90aW1lcyBNKV57XFxHTH0gXFxvdGltZXMgKFggXFxvdGltZXMgTSlee1xcR0x9IFxcb3RpbWVzIChZIFxcb3RpbWVzIE4pXntcXEdMfSJdLFsyLDAsIlxcbGVmdCggKFggXFxvdGltZXMgWCkgXFxvdGltZXMgKE0gXFxvdGltZXMgTSkgXFxyaWdodClee1xcR0x9IFxcb3RpbWVzIChZIFxcb3RpbWVzIE4pXntcXEdMfSJdLFswLDEsIihYIFxcb3RpbWVzIE0pXntcXEdMfSBcXG90aW1lcyAoKFggXFxvdGltZXMgWSkgXFxvdGltZXMgKE0gXFxvdGltZXMgTikpXntcXEdMfSAiXSxbMiwxLCIoKFggXFxvdGltZXMgWCBcXG90aW1lcyBZKSBcXG90aW1lcyAoTSBcXG90aW1lcyBNIFxcb3RpbWVzIE4pKV57XFxHTH0iXSxbMCwxXSxbMCwyXSxbMiwzXSxbMSwzXV0=
\[\begin{tikzcd}
	{(X \otimes M)^{\GL} \otimes (X \otimes M)^{\GL} \otimes (Y \otimes N)^{\GL}} && {\left( (X \otimes X) \otimes (M \otimes M) \right)^{\GL} \otimes (Y \otimes N)^{\GL}} \\
	{(X \otimes M)^{\GL} \otimes ((X \otimes Y) \otimes (M \otimes N))^{\GL} } && {((X \otimes X \otimes Y) \otimes (M \otimes M \otimes N))^{\GL}}
	\arrow[from=1-1, to=1-3]
	\arrow[from=1-1, to=2-1]
	\arrow[from=1-3, to=2-3]
	\arrow[from=2-1, to=2-3]
      \end{tikzcd} \qedhere
    \]
\end{proof}

\begin{lemma} \label{lem:sym-transform}
  Let $F$ be a vector (super)space and pick a partition $\lambda$. Set $X = S^\bullet(E \otimes F)$, $M = S^\bullet(E^*)$, and $N = S^\bullet(E^*) \otimes \bbs_\lambda(E^*)$.

  Then there is an isomorphism of algebras $\Phi^X(M) \cong S^\bullet(F)$, and $\Phi^X(N)$ is generated as an $S^\bullet(F)$-module by $\Phi^X(\bbs_\lambda(E^*))$, where $\bbs_\lambda(E^*)$ is identified with $S^0(E^*) \otimes \bbs_\lambda(E^*)$.
\end{lemma}

\begin{proof}
  Let $H$ be a vector space and set $N' = S^\bullet(E^* \otimes (H \oplus \bbc))$. If $\dim H \ge \ell(\lambda)$, then as a $\GL(H)$-representation, $N$ is the $\bbs_\lambda(H)$-isotypic component of $N'$. Notice that $ N'$ is naturally a polynomial $\GL(E)$-algebra and, by definition, $\Phi^X( N')$ is the ring of invariants of $ N'$ under $\GL(E)$. In particular, $\Phi^X(N')$ is a determinantal ring, i.e., the quotient of $T \coloneq S^\bullet(F \otimes (H \oplus \bbc))$ by the ideal generated by $\bigwedge^{\dim E+1} F \otimes \bigwedge^{\dim E + 1} (H \oplus \bbc)$. 

When $H=0$, the algebra structure on $\Phi^X(M)$ is the determinantal ring above, which is just the polynomial ring $S^\bullet(F)$ since $\dim E + 1 > 1$. For general $H$, the algebra $T$ is free over the subalgebra $S^\bullet(F)$, and so the $\bbs_\lambda(H)$-isotypic component of $S^\bullet(F \otimes (H\oplus \bbc))$, being a direct summand, is also free.  For degree reasons, it must be generated by $\bbs_\lambda(F)$, which implies that $N$ is also generated by the image of $\bbs_\lambda(F)=\Phi^X(\bbs_\lambda(E^*))$ in $\Phi^X(N)$.
\end{proof}

\begin{lemma}\label{lem:PhiConvertsIntoAModule}
    Let $(V,q,V_\circ)$ be a framed orthogonal space with $\dim V = n$. Set $A \coloneq S^\bullet (V) / (q)$. If $\dim E > \ell(\lambda)$, then there is an isomorphism of $A$-modules
    \[
      \Phi^{Z_{V,E}} (S^\bullet (E^*) \otimes \bbs_\lambda (E^*) ) \cong A \otimes \bbs_\lambda^A
    \]
    where the right hand side is the free $A$-module with generators $\bbs_\lambda^A$.
\end{lemma}

\begin{proof}
  Note that it is clear that \emph{as vector spaces}, there is an isomorphism
  \[
    \Phi^{Z_{V,E}} (S^\bullet (E^*) \otimes \bbs_\lambda (E^*) ) \cong A \otimes \bbs_\lambda^A.
  \]
  So we have to show that the result of applying $\Phi^{Z_{V,E}}$ is a free $A$-module. To do this, we will first prove that the result holds when $\dim V \gg 0$, then deduce the result in general for odd-dimensional spaces. 
    
  First, consider $S^\bullet(V \otimes E)$ as a quotient of $\uU(\g_{V,E})$ by the 2-sided ideal generated by $\bigwedge^2 E$. All irreducible $\GL(E)$-subrepresentations in this ideal have highest weights indexed by partitions with at least $2$ parts, so
  \[
    \Phi^{S^\bullet(V \otimes E)}(S^\bullet (E^*)) \cong \Phi^{\uU(\g_{V,E})}(S^\bullet(E^*)).
  \]
  Now it follows from Lemma~\ref{lem:sym-transform} that $\Phi^{S^\bullet(V \otimes E)}(S^\bullet (E^*)) \cong S^\bullet(V)$.

  \medskip

    \textbf{The Large Rank Case:} We will use the notation from \S\ref{ss:geom}. Assume that $q$ has full rank and $\dim V = 2d$, where $d > \dim E$. Set $X := \igr (d, V)$ and consider the bundle 
    \[
      \eta_{V,E} = \uU (\underline{\g_{V,E}} ) \otimes_{\uU(\R \otimes E)} \OO_X \cong S^\bullet (\R^* \otimes E) \otimes S^\bullet (\bigwedge^2 E),
    \]
    where the rightmost isomorphism is only an isomorphism as $\OO_X$-modules. Recall that $\eta_{V,E}$ is viewed as a left $\uU (\underline{\g_{V,E}} )$-module. We have
    \[
      \Phi^{\eta_{V,E}} (S^\bullet (E^* \otimes (H \oplus \bbc))) \cong \bigoplus_{\lambda, \mu, \nu} \bbs_\nu (\R^*)^{\oplus c_{\nu,(2\mu)^T}^\lambda} \otimes \bbs_\lambda (H \oplus \bbc).
    \]
    Using the Pieri rule, we have that as a $\GL(H)$-representation, the $\bbs_\lambda (H)$-multiplicity space is $\bigoplus_{\lambda', \mu, \nu} \bbs_\nu(\R^*)^{\oplus c^{\lambda'}_{\nu, (2\mu)^T}}$ where $\lambda'$ ranges over all partitions containing $\lambda$ such that $\lambda'/\lambda$ is a horizontal strip. For fixed $\lambda'$, using the Borel--Weil theorem (\cite[Corollary 4.3.9]{weyman2003}) and Proposition~\ref{prop:stable-decomp}, we have
    \[
      {\rm H}^0(X, \bigoplus_{\mu, \nu} \bbs_\nu(\R^*)^{\oplus c^{\lambda'}_{\nu, (2\mu)^T}}) \cong \bbs^A_{\lambda'}.
    \]
    Hence if we sum over all $\lambda'$, the global sections of the $\bbs_\lambda(H)$-multiplicity space is isomorphic to $A \otimes \bbs_\lambda^A$ by the generalized Pieri rule (Proposition~\ref{cor:fundamentalIdentities}), so it remains to show that the $\bbs_\lambda (H)$-multiplicity space of $\Phi^{\eta_{V,E}} (S^\bullet (E^* \otimes (H \oplus \bbc)))$ is generated in minimal degree as an $S^\bullet(V) \otimes \OO_X$-module. 

    Tensoring with $S^\bullet (E^* \otimes (H \oplus \bbc))$ and taking $\GL (E)$-invariants is a monoidal transformation that replaces any copy of $E$ with a copy of $H \oplus \bbc$, and thus 
    \[
      \Phi^{\uU (\underline{\g_{V,E}})} (S^\bullet (E^* \otimes (H \oplus \bbc)) \to \Phi^{\eta_{V,E}} (S^\bullet (E^* \otimes (H \oplus \bbc)))
    \]
    being surjective implies that
    \[
      \uU (\underline{\g_{V , H \oplus \bbc}}) \to \eta_{V , H \oplus \bbc}
    \]
    is surjective.
    
    The enveloping algebra $\uU (\underline{\g_{V,H \oplus \bbc}})$ is generated in degree $1$ and thus so is the quotient $\eta_{V, H \oplus \bbc}$ (as a left $\uU (\underline{\g_{V,H \oplus \bbc}})$-module). This implies that all $\bbs_\lambda (H)$-multiplicity spaces are also generated in minimal degree, and thus upon taking global sections the $S^\bullet (V)$-module $\Phi^{Z_{V,E}} (S^\bullet (E^*) \otimes \bbs_\lambda (E^*)) \cong A \otimes \bbs_\lambda^A$ must be generated in degree $|\lambda|$.  
    
    Finally, by examining the Hilbert function, the $S^\bullet (V)$-module $A \otimes \bbs_\lambda^A$ is annihilated by some quadric; by equivariance this quadric must be $q$, whence $\Phi^{Z_{V,E}} (S^\bullet (E^*) \otimes \bbs_\lambda (E^*)) \cong A \otimes \bbs_\lambda^A$ is a free $A$-module.

    \medskip

    \textbf{General Odd Case:} Assume now that $\dim V$ is odd, $\dim E$ is arbitrary, and $q$ is full rank. Let $A = S^\bullet (V) / (q)$ and note that there exists a graded surjection $A' \twoheadrightarrow A$ from $A' := S^\bullet (V') / (q')$ where $\dim V' = 2d$ is even, $d> \dim E$, and $q'$ is a full rank quadric. This surjection induces a natural graded surjection $Z_{V',E} \twoheadrightarrow Z_{V,E}$ of left $\uU (\g_{V',E})$-modules (where $Z_{V,E}$ is viewed as a left $\uU( \g_{V',E})$-module via the surjection $\uU( \g_{V',E} )\twoheadrightarrow\uU( \g_{V,E} )$), and since taking $\GL$-invariants is an exact functor there is an induced surjection of $S^\bullet (V')$-modules
    $$A' \otimes \bbs_\lambda^{A'} \twoheadrightarrow \Phi^{Z_{V,E}} (S^\bullet (E^*) \otimes \bbs_\lambda (E^*)),$$
    where the $S^\bullet (V')$-module structure on the target is induced by the surjection $S^\bullet (V') \twoheadrightarrow S^\bullet (V)$. Since $\Phi^{Z_{V,E}} (S^\bullet (E^*) \otimes \bbs_\lambda (E^*))$ is a quotient of a module generated in the minimal possible degree, it itself must be generated in minimal degree and thus $\Phi^{Z_{V,E}} (S^\bullet (E^*) \otimes \bbs_\lambda (E^*))$ is isomorphic to the free $A$-module $A \otimes \bbs_\lambda^A$. 
    
    Finally, to handle the case of an arbitrary framed orthogonal space $(V, q, V_\circ)$ equipped with an odd-rank quadric, we simply note that the framed structure induces an isomorphism
    $$A \coloneqq S^\bullet (V) / (q) \cong S^\bullet (V_\circ)/(q) \otimes S^\bullet (V / V_\circ),$$
    which in turn gives an isomorphism of functors
    \[
      \Phi^{Z_{V,E}} \cong \Phi^{Z_{V_\circ,E}} \otimes \Phi^{S^\bullet (V/V_\circ \otimes E)},
    \]
    and thus the result follows immediately from the full rank quadric case.
\end{proof}

\begin{remark}\label{rk:noMonoidalStructure}
  To explain why we had to use this indirect proof, we show that there is {\bf no monoidal linear functor} from $\GL(E)$-representations to ${\rm O}(V)$-representations (we will assume $q$ has full rank) that sends $\bbs_\lambda(E)$ to $\bbs_\lambda^A$. (We thank Andrew Snowden for this observation and for letting us include it.) Suppose that such a functor $\Psi$ exists, let
  \[
    \sigma \colon E \otimes E \to E \otimes E, \qquad \sigma(e \otimes e') = e' \otimes e
  \]
  be the swapping map, and let $\tau = \Psi(\sigma)$. By definition, $\bigwedge^2 E$ is the $-1$ eigenspace of $\sigma$ and $S^2 E$ is the $+1$-eigenspace of $\sigma$. Then $\bigwedge^2 V \oplus {\rm span}(q)$ is the $-1$ eigenspace of $\tau$ and its $+1$-eigenspace is $A_2$, thought of as the ${\rm O}(V)$-complement of ${\rm span}(q)$ in $S^2(V)$. This determines a formula for $\tau$.

  Next, let $\sigma_1 \colon E^{\otimes 3} \to E^{\otimes 3}$ denote $\sigma \otimes 1$ and let $\sigma_2 = 1 \otimes \sigma$. Then they satisfy the braid relation: $\sigma_1 \sigma_2 \sigma_1 = \sigma_2 \sigma_1 \sigma_2$. Since they are defined using the monoidal structure, we see that $\tau$ must satisfy a similar braid relation $\tau_1 \tau_2 \tau_1 = \tau_2 \tau_1 \tau_2$. But this identity fails.

  For simplicity, we illustrate this when $\dim V = 2$; let $q = x^2 + y^2$. Then the $-1$-eigenspace of $\tau$ is ${\rm span}(x \otimes y - y \otimes x, x\otimes x + y \otimes y)$ while the $+1$-eigenspace is ${\rm span}(x \otimes x - y \otimes y, x\otimes y + y \otimes x)$. Hence $\tau(x \otimes x) = - y \otimes y$ and $\tau(x \otimes y) = y \otimes x$ (the rest is determined by $\tau^2=1$). In particular,
  \[
    \tau_1\tau_2\tau_1(x \otimes x \otimes x) = -x \otimes y \otimes y, \qquad \tau_2 \tau_1 \tau_2(x \otimes x \otimes x) = - y \otimes y \otimes x. \qedhere
  \]
\end{remark}

\begin{remark}
  One of the obstacles to our construction is that $Z_{V,E}$ does not have any $\GL(E)$-equivariant algebra structure in general. As discussed in Remark~\ref{rmk:comm-degen}, $Z_{V,E}$ does have a flat degeneration to a commutative algebra $Z'$, namely a certain quotient of $S^\bullet(\g_{V,E})$. However, $Z'$ is not generated in degree 1 in general, so that the functor $\Phi^{Z'}$ will not behave well with respect to free modules as above. Thus the Lie structure on these objects is essential for the functor $\Phi^{Z_{V,E}}$ to behave as expected.
\end{remark}

\section{Pure free resolutions over quadric hypersurface rings} \label{sec:purefree}

In this section, we illustrate the relationship between (equivariant) total positivity and the construction of pure free resolutions. In particular, we show that if one modifies the definition of a Schur module to be ``with respect to" the underlying algebra, identical constructions to that of Eisenbud--Fl\o ystad--Weyman \cite{eisenbud2011existence} may be used to construct pure free resolutions over quadric hypersurfaces and (coordinate rings of) rational normal curves. 

\subsection{EFW complexes} \label{sec:EFW}

In this section, we briefly recall the EFW complexes constructed in \cite[Theorem 0.1]{eisenbud2011existence}. Let $E$ be a vector superspace and set $S = S^\bullet(E)$.

\begin{theorem}[{\cite{eisenbud2011existence}}]\label{thm:originalPieriResolutions}
  Given any sequence of positive integers $e_1 , e_2,\dots$ such that $e_i = 1$ for all $i \gg 0$, set $d_i = e_1+\cdots+e_i$. There exists a $\GL(E)$-equivariant resolution of the form
    \[
      \cdots \to S(-d_n) \otimes  \bbs^{\lambda^{(n)}} \to S(-d_{n-1}) \otimes  \bbs_{\lambda^{(n-1)}}(E) \to \cdots \to S \otimes  \bbs_{\lambda^{(0)}}(E).
    \]
  \end{theorem}

  These resolutions are functorial for linear maps $E \to E'$ and so it is also possible to phrase this as a resolution of polynomial functors (and thereby bypassing the initial need to fix a space $E$).
Since the space $E$ is arbitrary, we can replace it by its dual $E^*$. This is how we will use it below.

A graphical way to get the partitions $\lambda^{(0)}, \lambda^{(1)} , \dots , \lambda^{(n)} , \dots $ is to draw the (infinite) border strip with row lengths given by the shifts $e_1, e_2 , \dots$ (reading from top to bottom). This region above the border strip is the Young diagram for a partition $\lambda$ and we set $\lambda^{(0)} \coloneq \lambda$. The subsequent partitions are obtained by adding the rows of the above border strip to $\lambda$ starting from top to bottom. The following example illustrates this (truncating everything after row 4):
    \tikzset{every picture/.style={line width=0.75pt}} %set default line width to 0.75pt        
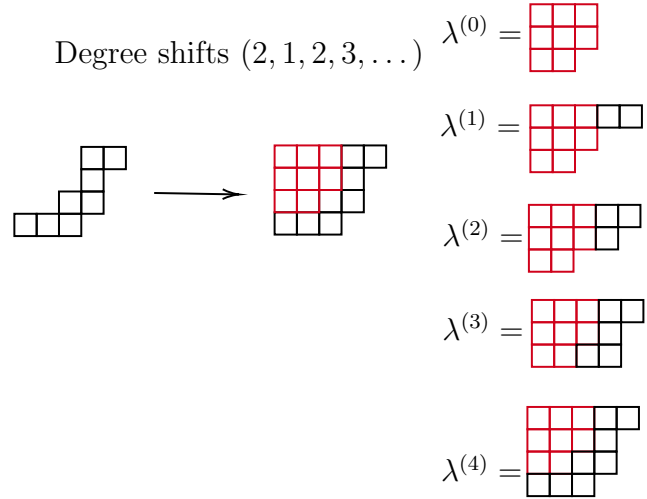
\begin{center}
\begin{tikzpicture}[x=0.75pt,y=0.75pt, yscale=-.75, xscale=.75]
%uncomment if require: \path (0,375); %set diagram left start at 0, and has height of 375

%Shape: Rectangle [id:dp15987477514804982] 
\draw   (79,164) -- (94,164) -- (94,179) -- (79,179) -- cycle ;
%Shape: Rectangle [id:dp5312315848714904] 
\draw   (94,164) -- (109,164) -- (109,179) -- (94,179) -- cycle ;
%Shape: Rectangle [id:dp9076990153228808] 
\draw   (109,164) -- (124,164) -- (124,179) -- (109,179) -- cycle ;
%Shape: Rectangle [id:dp7501003205549654] 
\draw   (109,149) -- (124,149) -- (124,164) -- (109,164) -- cycle ;
%Shape: Rectangle [id:dp490153520962499] 
\draw   (124,149) -- (139,149) -- (139,164) -- (124,164) -- cycle ;
%Shape: Rectangle [id:dp7152507923102498] 
\draw   (124,134) -- (139,134) -- (139,149) -- (124,149) -- cycle ;
%Shape: Rectangle [id:dp5431919831620156] 
\draw   (124,119) -- (139,119) -- (139,134) -- (124,134) -- cycle ;
%Shape: Rectangle [id:dp7512250966287404] 
\draw   (139,119) -- (154,119) -- (154,134) -- (139,134) -- cycle ;
%Shape: Rectangle [id:dp1303847015652142] 
\draw   (254,163) -- (269,163) -- (269,178) -- (254,178) -- cycle ;
%Shape: Rectangle [id:dp2990165369530846] 
\draw   (269,163) -- (284,163) -- (284,178) -- (269,178) -- cycle ;
%Shape: Rectangle [id:dp7509260986994017] 
\draw   (284,163) -- (299,163) -- (299,178) -- (284,178) -- cycle ;
%Shape: Rectangle [id:dp760838664353263] 
\draw   (284,148) -- (299,148) -- (299,163) -- (284,163) -- cycle ;
%Shape: Rectangle [id:dp06932140294144529] 
\draw   (299,148) -- (314,148) -- (314,163) -- (299,163) -- cycle ;
%Shape: Rectangle [id:dp7501719088512795] 
\draw   (299,133) -- (314,133) -- (314,148) -- (299,148) -- cycle ;
%Shape: Rectangle [id:dp4274816200401088] 
\draw   (299,118) -- (314,118) -- (314,133) -- (299,133) -- cycle ;
%Shape: Rectangle [id:dp23756892313608668] 
\draw   (314,118) -- (329,118) -- (329,133) -- (314,133) -- cycle ;
%Straight Lines [id:da8410369083471714] 
\draw    (173,150) -- (228,150.96) ;
\draw [shift={(230,151)}, rotate = 181.01] [color={rgb, 255:red, 0; green, 0; blue, 0 }  ][line width=0.75]    (10.93,-3.29) .. controls (6.95,-1.4) and (3.31,-0.3) .. (0,0) .. controls (3.31,0.3) and (6.95,1.4) .. (10.93,3.29)   ;
%Shape: Rectangle [id:dp7251857622699307] 
\draw  [color={rgb, 255:red, 208; green, 2; blue, 27 }  ,draw opacity=1 ] (284,118) -- (299,118) -- (299,133) -- (284,133) -- cycle ;
%Shape: Rectangle [id:dp5143192663428842] 
\draw  [color={rgb, 255:red, 208; green, 2; blue, 27 }  ,draw opacity=1 ] (269,118) -- (284,118) -- (284,133) -- (269,133) -- cycle ;
%Shape: Rectangle [id:dp6475679965225638] 
\draw  [color={rgb, 255:red, 208; green, 2; blue, 27 }  ,draw opacity=1 ] (254,118) -- (269,118) -- (269,133) -- (254,133) -- cycle ;
%Shape: Rectangle [id:dp4275294048717342] 
\draw  [color={rgb, 255:red, 208; green, 2; blue, 27 }  ,draw opacity=1 ] (284,133) -- (299,133) -- (299,148) -- (284,148) -- cycle ;
%Shape: Rectangle [id:dp23225680389174075] 
\draw  [color={rgb, 255:red, 208; green, 2; blue, 27 }  ,draw opacity=1 ] (269,133) -- (284,133) -- (284,148) -- (269,148) -- cycle ;
%Shape: Rectangle [id:dp9230199622411257] 
\draw  [color={rgb, 255:red, 208; green, 2; blue, 27 }  ,draw opacity=1 ] (254,133) -- (269,133) -- (269,148) -- (254,148) -- cycle ;
%Shape: Rectangle [id:dp08325684281886803] 
\draw  [color={rgb, 255:red, 208; green, 2; blue, 27 }  ,draw opacity=1 ] (269,148) -- (284,148) -- (284,163) -- (269,163) -- cycle ;
%Shape: Rectangle [id:dp5533816516971719] 
\draw  [color={rgb, 255:red, 208; green, 2; blue, 27 }  ,draw opacity=1 ] (254,148) -- (269,148) -- (269,163) -- (254,163) -- cycle ;
%Shape: Rectangle [id:dp031445284929797035] 
\draw  [color={rgb, 255:red, 208; green, 2; blue, 27 }  ,draw opacity=1 ] (456,91) -- (471,91) -- (471,106) -- (456,106) -- cycle ;
%Shape: Rectangle [id:dp5306771762063789] 
\draw  [color={rgb, 255:red, 208; green, 2; blue, 27 }  ,draw opacity=1 ] (441,91) -- (456,91) -- (456,106) -- (441,106) -- cycle ;
%Shape: Rectangle [id:dp6936403802912172] 
\draw  [color={rgb, 255:red, 208; green, 2; blue, 27 }  ,draw opacity=1 ] (426,91) -- (441,91) -- (441,106) -- (426,106) -- cycle ;
%Shape: Rectangle [id:dp0983271819495477] 
\draw  [color={rgb, 255:red, 208; green, 2; blue, 27 }  ,draw opacity=1 ] (456,106) -- (471,106) -- (471,121) -- (456,121) -- cycle ;
%Shape: Rectangle [id:dp4096917809661844] 
\draw  [color={rgb, 255:red, 208; green, 2; blue, 27 }  ,draw opacity=1 ] (441,106) -- (456,106) -- (456,121) -- (441,121) -- cycle ;
%Shape: Rectangle [id:dp5746576650245521] 
\draw  [color={rgb, 255:red, 208; green, 2; blue, 27 }  ,draw opacity=1 ] (426,106) -- (441,106) -- (441,121) -- (426,121) -- cycle ;
%Shape: Rectangle [id:dp42970723937556476] 
\draw  [color={rgb, 255:red, 208; green, 2; blue, 27 }  ,draw opacity=1 ] (441,121) -- (456,121) -- (456,136) -- (441,136) -- cycle ;
%Shape: Rectangle [id:dp24943491945852236] 
\draw  [color={rgb, 255:red, 208; green, 2; blue, 27 }  ,draw opacity=1 ] (426,121) -- (441,121) -- (441,136) -- (426,136) -- cycle ;
%Shape: Rectangle [id:dp7988861262267588] 
\draw  [color={rgb, 255:red, 208; green, 2; blue, 27 }  ,draw opacity=1 ] (455,158) -- (470,158) -- (470,173) -- (455,173) -- cycle ;
%Shape: Rectangle [id:dp3803145276548108] 
\draw  [color={rgb, 255:red, 208; green, 2; blue, 27 }  ,draw opacity=1 ] (440,158) -- (455,158) -- (455,173) -- (440,173) -- cycle ;
%Shape: Rectangle [id:dp3210000991582469] 
\draw  [color={rgb, 255:red, 208; green, 2; blue, 27 }  ,draw opacity=1 ] (425,158) -- (440,158) -- (440,173) -- (425,173) -- cycle ;
%Shape: Rectangle [id:dp35129194456220936] 
\draw  [color={rgb, 255:red, 208; green, 2; blue, 27 }  ,draw opacity=1 ] (455,173) -- (470,173) -- (470,188) -- (455,188) -- cycle ;
%Shape: Rectangle [id:dp7617296915340808] 
\draw  [color={rgb, 255:red, 208; green, 2; blue, 27 }  ,draw opacity=1 ] (440,173) -- (455,173) -- (455,188) -- (440,188) -- cycle ;
%Shape: Rectangle [id:dp12004143201505024] 
\draw  [color={rgb, 255:red, 208; green, 2; blue, 27 }  ,draw opacity=1 ] (425,173) -- (440,173) -- (440,188) -- (425,188) -- cycle ;
%Shape: Rectangle [id:dp13399028776438882] 
\draw  [color={rgb, 255:red, 208; green, 2; blue, 27 }  ,draw opacity=1 ] (440,188) -- (455,188) -- (455,203) -- (440,203) -- cycle ;
%Shape: Rectangle [id:dp26504832467626094] 
\draw  [color={rgb, 255:red, 208; green, 2; blue, 27 }  ,draw opacity=1 ] (425,188) -- (440,188) -- (440,203) -- (425,203) -- cycle ;
%Shape: Rectangle [id:dp39182366751777664] 
\draw  [color={rgb, 255:red, 208; green, 2; blue, 27 }  ,draw opacity=1 ] (457,222) -- (472,222) -- (472,237) -- (457,237) -- cycle ;
%Shape: Rectangle [id:dp4271363607253731] 
\draw  [color={rgb, 255:red, 208; green, 2; blue, 27 }  ,draw opacity=1 ] (442,222) -- (457,222) -- (457,237) -- (442,237) -- cycle ;
%Shape: Rectangle [id:dp4604564294972764] 
\draw  [color={rgb, 255:red, 208; green, 2; blue, 27 }  ,draw opacity=1 ] (427,222) -- (442,222) -- (442,237) -- (427,237) -- cycle ;
%Shape: Rectangle [id:dp4671996157319327] 
\draw  [color={rgb, 255:red, 208; green, 2; blue, 27 }  ,draw opacity=1 ] (457,237) -- (472,237) -- (472,252) -- (457,252) -- cycle ;
%Shape: Rectangle [id:dp021178368045973395] 
\draw  [color={rgb, 255:red, 208; green, 2; blue, 27 }  ,draw opacity=1 ] (442,237) -- (457,237) -- (457,252) -- (442,252) -- cycle ;
%Shape: Rectangle [id:dp4468923809875325] 
\draw  [color={rgb, 255:red, 208; green, 2; blue, 27 }  ,draw opacity=1 ] (427,237) -- (442,237) -- (442,252) -- (427,252) -- cycle ;
%Shape: Rectangle [id:dp983892267999398] 
\draw  [color={rgb, 255:red, 208; green, 2; blue, 27 }  ,draw opacity=1 ] (442,252) -- (457,252) -- (457,267) -- (442,267) -- cycle ;
%Shape: Rectangle [id:dp5004707619542184] 
\draw  [color={rgb, 255:red, 208; green, 2; blue, 27 }  ,draw opacity=1 ] (427,252) -- (442,252) -- (442,267) -- (427,267) -- cycle ;
%Shape: Rectangle [id:dp5519317260459895] 
\draw  [color={rgb, 255:red, 208; green, 2; blue, 27 }  ,draw opacity=1 ] (454,294) -- (469,294) -- (469,309) -- (454,309) -- cycle ;
%Shape: Rectangle [id:dp03675755485619536] 
\draw  [color={rgb, 255:red, 208; green, 2; blue, 27 }  ,draw opacity=1 ] (439,294) -- (454,294) -- (454,309) -- (439,309) -- cycle ;
%Shape: Rectangle [id:dp5331897793166054] 
\draw  [color={rgb, 255:red, 208; green, 2; blue, 27 }  ,draw opacity=1 ] (424,294) -- (439,294) -- (439,309) -- (424,309) -- cycle ;
%Shape: Rectangle [id:dp9710457890794626] 
\draw  [color={rgb, 255:red, 208; green, 2; blue, 27 }  ,draw opacity=1 ] (454,309) -- (469,309) -- (469,324) -- (454,324) -- cycle ;
%Shape: Rectangle [id:dp5419479582132376] 
\draw  [color={rgb, 255:red, 208; green, 2; blue, 27 }  ,draw opacity=1 ] (439,309) -- (454,309) -- (454,324) -- (439,324) -- cycle ;
%Shape: Rectangle [id:dp8630625823770928] 
\draw  [color={rgb, 255:red, 208; green, 2; blue, 27 }  ,draw opacity=1 ] (424,309) -- (439,309) -- (439,324) -- (424,324) -- cycle ;
%Shape: Rectangle [id:dp18103080619142475] 
\draw  [color={rgb, 255:red, 208; green, 2; blue, 27 }  ,draw opacity=1 ] (439,324) -- (454,324) -- (454,339) -- (439,339) -- cycle ;
%Shape: Rectangle [id:dp29107340805435555] 
\draw  [color={rgb, 255:red, 208; green, 2; blue, 27 }  ,draw opacity=1 ] (424,324) -- (439,324) -- (439,339) -- (424,339) -- cycle ;
%Shape: Rectangle [id:dp10676915887985294] 
\draw   (471,91) -- (486,91) -- (486,106) -- (471,106) -- cycle ;
%Shape: Rectangle [id:dp29674726879224567] 
\draw   (486,91) -- (501,91) -- (501,106) -- (486,106) -- cycle ;
%Shape: Rectangle [id:dp27251319336766366] 
\draw   (470,158) -- (485,158) -- (485,173) -- (470,173) -- cycle ;
%Shape: Rectangle [id:dp13210302131728668] 
\draw   (485,158) -- (500,158) -- (500,173) -- (485,173) -- cycle ;
%Shape: Rectangle [id:dp2601934415402376] 
\draw   (470,173) -- (485,173) -- (485,188) -- (470,188) -- cycle ;
%Shape: Rectangle [id:dp9150783304408798] 
\draw   (472,222) -- (487,222) -- (487,237) -- (472,237) -- cycle ;
%Shape: Rectangle [id:dp7254836362507362] 
\draw   (472,237) -- (487,237) -- (487,252) -- (472,252) -- cycle ;
%Shape: Rectangle [id:dp29187261946079146] 
\draw   (472,252) -- (487,252) -- (487,267) -- (472,267) -- cycle ;
%Shape: Rectangle [id:dp3493644947048804] 
\draw   (457,252) -- (472,252) -- (472,267) -- (457,267) -- cycle ;
%Shape: Rectangle [id:dp3123424417940992] 
\draw   (469,294) -- (484,294) -- (484,309) -- (469,309) -- cycle ;
%Shape: Rectangle [id:dp23820764577229903] 
\draw   (487,222) -- (502,222) -- (502,237) -- (487,237) -- cycle ;
%Shape: Rectangle [id:dp49688990473760364] 
\draw   (484,294) -- (499,294) -- (499,309) -- (484,309) -- cycle ;
%Shape: Rectangle [id:dp1763720741112591] 
\draw   (469,309) -- (484,309) -- (484,324) -- (469,324) -- cycle ;
%Shape: Rectangle [id:dp06923511932730952] 
\draw   (469,324) -- (484,324) -- (484,339) -- (469,339) -- cycle ;
%Shape: Rectangle [id:dp7548889217437247] 
\draw   (454,324) -- (469,324) -- (469,339) -- (454,339) -- cycle ;
%Shape: Rectangle [id:dp7242876632615709] 
\draw   (454,339) -- (469,339) -- (469,354) -- (454,354) -- cycle ;
%Shape: Rectangle [id:dp21029874228323986] 
\draw   (439,339) -- (454,339) -- (454,354) -- (439,354) -- cycle ;
%Shape: Rectangle [id:dp13908916249192527] 
\draw   (424,339) -- (439,339) -- (439,354) -- (424,354) -- cycle ;
%Shape: Rectangle [id:dp8951641317809869] 
\draw  [color={rgb, 255:red, 208; green, 2; blue, 27 }  ,draw opacity=1 ] (456,23) -- (471,23) -- (471,38) -- (456,38) -- cycle ;
%Shape: Rectangle [id:dp27277195912482344] 
\draw  [color={rgb, 255:red, 208; green, 2; blue, 27 }  ,draw opacity=1 ] (441,23) -- (456,23) -- (456,38) -- (441,38) -- cycle ;
%Shape: Rectangle [id:dp04009266698046199] 
\draw  [color={rgb, 255:red, 208; green, 2; blue, 27 }  ,draw opacity=1 ] (426,23) -- (441,23) -- (441,38) -- (426,38) -- cycle ;
%Shape: Rectangle [id:dp29828071771025955] 
\draw  [color={rgb, 255:red, 208; green, 2; blue, 27 }  ,draw opacity=1 ] (456,38) -- (471,38) -- (471,53) -- (456,53) -- cycle ;
%Shape: Rectangle [id:dp5353237162734801] 
\draw  [color={rgb, 255:red, 208; green, 2; blue, 27 }  ,draw opacity=1 ] (441,38) -- (456,38) -- (456,53) -- (441,53) -- cycle ;
%Shape: Rectangle [id:dp42470546529469466] 
\draw  [color={rgb, 255:red, 208; green, 2; blue, 27 }  ,draw opacity=1 ] (426,38) -- (441,38) -- (441,53) -- (426,53) -- cycle ;
%Shape: Rectangle [id:dp2287582006466784] 
\draw  [color={rgb, 255:red, 208; green, 2; blue, 27 }  ,draw opacity=1 ] (441,53) -- (456,53) -- (456,68) -- (441,68) -- cycle ;
%Shape: Rectangle [id:dp013563553479509327] 
\draw  [color={rgb, 255:red, 208; green, 2; blue, 27 }  ,draw opacity=1 ] (426,53) -- (441,53) -- (441,68) -- (426,68) -- cycle ;

% Text Node
\draw (364,28.4) node [anchor=north west][inner sep=0.75pt]    {$\lambda ^{( 0)} =$};
% Text Node
\draw (363,91.4) node [anchor=north west][inner sep=0.75pt]    {$\lambda ^{( 1)} =$};
% Text Node
\draw (365,165.4) node [anchor=north west][inner sep=0.75pt]    {$\lambda ^{( 2)} =$};
% Text Node
\draw (365,228.4) node [anchor=north west][inner sep=0.75pt]    {$\lambda ^{( 3)} =$};
% Text Node
\draw (104,44.4) node [anchor=north west][inner sep=0.75pt]    {$\text{Degree shifts } ( 2,1,2,3,\dots)$};
% Text Node
\draw (364,323.4) node [anchor=north west][inner sep=0.75pt]    {$\lambda ^{( 4)} =$};
\end{tikzpicture}
\end{center}

\begin{remark} \label{rmk:EFW-combinatorics}
  Suppose $E$ is a vector space with $\dim E = n$. If $e_n > 1$, then according to the pattern above, $\lambda_{n+1} >0$, so that $\bbs_{\lambda^{(i)}}(E)=0$ for all $i$, and the EFW complex is 0. Hence to get nonzero complexes, one needs to choose $\dim E$ large enough so that $e_i=1$ for all $i \ge \dim E$.
  Furthermore, with such a choice of $\dim E$, we see that $\dim E > \ell(\lambda^{(i)})$ whenever $i < \dim E$.
\end{remark}

\subsection{EFW complexes over odd quadric hypersurface rings}

In this subsection, we show that a construction nearly identical to that of Eisenbud--Fl\o ystad--Weyman \cite{eisenbud2011existence} can be used to construct equivariant pure free resolutions over the odd quadric hypersurface ring, where one replaces classical Schur functors with the quadric Schur functors. 

\begin{theorem}\label{thm:quadricPieriResolutions}
  Let $(V,q,V_\circ)$ be a framed orthogonal space equipped with an odd-rank quadric with $\dim V = n$, and let $G = {\rm O}(V,q,V_\circ)$.
  Let $A = S^\bullet (V) / (q)$ be the associated quadric hypersurface. Pick a sequence of positive integers $e_1 , \dots , e_{n}$ and set $d_i = \sum_{j=1}^i e_j$.
    \begin{enumerate}
        \item If $e_{n} > 1$, there is a finite length $G$-equivariant pure free resolution 
          \[
            0 \to A (-d_{n-1}) \otimes \bbs_{\lambda^{(n-1)}}^A \to  \cdots \to A (-d_{1}) \otimes \bbs_{\lambda^{(1)}}^A \to A \otimes \bbs_{\lambda^{(0)}}^A
          \]
          with degree sequence $0,d_1 , d_2 , \dots , d_{n-1}$. 

      \item If $e_{n} = 1$, then there is an infinite length $G$-equivariant pure free resolution
          \begin{align*}
            \cdots \to A (-d_{n-1}-2) \otimes \bbs_{(\lambda^{(n-1)},1,1)}^A \to A (-d_{n-1}-1) \otimes \bbs_{(\lambda^{(n-1)},1)}^A \to \\
            A (-d_{n-1}) \otimes \bbs_{\lambda^{(n-1)}}^A
        \to A (-d_{n-2}) \otimes \bbs_{\lambda^{(n-2)}}^A \to \cdots \to A (-d_{1}) \otimes \bbs_{\lambda^{(1)}}^A \to A \otimes \bbs_{\lambda^{(0)}}^A
          \end{align*}
          with infinite degree sequence $0,d_1 , \dots , d_{n-1} , d_{n-1}+1 ,d_{n-1}+2 , \dots$.
        \end{enumerate}
              In both cases, the module being resolved has finite length.
\end{theorem}

\begin{proof}
  Define $e_{n+i}=1$ for all $i>0$.
  Pick $E$ so that $m = \dim E \ge n+1$.   Recall the definition of $\Phi^{Z_{V,E}}$ from Definition \ref{def:thePhiFunctors}, where $Z_{V,E} = \uU (\g_{V,E}) / (S^2 (E))$.
  From Remark~\ref{rmk:EFW-combinatorics}, we get that $m > \ell(\lambda^{(i)})$ if $i< m$, and then Lemma~\ref{lem:PhiConvertsIntoAModule} gives an isomorphism of $A$-modules
  \[
    \Phi^{Z_{V,E}}(S^\bullet(E^*)(-d_i) \otimes \bbs_{\lambda^{(i)}}(E^*)) \cong A(-d_i) \otimes \bbs_{\lambda^{(i)}}^A.
  \]
  Since $\Phi^{Z_{V,E}}$ is induced by taking $\GL (E)$-invariants, it is an exact functor. Applying $\Phi^{Z_{V,E}}$ to the first $m$ terms of the EFW complex of Theorem \ref{thm:originalPieriResolutions} gives an exact complex
  \[
    A(-d_{m-1}) \otimes  \bbs_{\lambda^{(m-1)}}^A \to A(-d_{m-2}) \otimes  \bbs_{\lambda^{(m-2)}}^A \to \cdots \to A \otimes  \bbs_{\lambda^{(0)}}^A.
  \]
  Next, if we replace $E$ by $E \oplus \bbc$, the above construction returns 
  \[
    A(-d_{m}) \otimes  \bbs_{\lambda^{(m)}}^A \to A(-d_{m-1}) \otimes  \bbs_{\lambda^{(m-1)}}^A \to \cdots \to A \otimes  \bbs_{\lambda^{(0)}}^A,
  \]
  which extends the previous exact complex. We thus obtain a directed system of complexes, and the colimit of this system is our desired construction.

  For all $i \ge n$, we have $\lambda^{(i)}_{n} = e_n$. The two cases of the theorem are then a consequence of Proposition~\ref{prop:quadric-schurdim}.
  The last statement follows from the fact that the EFW complex is resolving a finite length module.
\end{proof}

\begin{remark}
    The pattern that all pure resolutions are either finite or have an infinite ``linear tail'' is expected by the general theory of resolutions over hypersurfaces, where the differentials of these tails are given by homogeneous minimal matrix factorizations of the quadric $q$. Since $q$ is quadratic any such matrix factorization will have linear entries.
\end{remark}

\begin{example}
    Consider the case of choosing degree shifts of the form $d , 1 , 1 , \dots$. This situation falls into the case $(2)$ of Theorem \ref{thm:quadricPieriResolutions} and we get resolutions of the form
    \[
      \cdots \to A(-d-i) \otimes \bbs_{(d,1^i)}^A \to  \cdots \to A(-d-1) \otimes \bbs_{(d,1)}^A \to A(-d) \otimes A_d \to A.
    \]
    In this special case, these resolutions can be constructed in a characteristic-free manner as the Priddy complex associated to the Koszul modules $A_+^d$ for $d > 0$ (where $A_+$ denotes the irrelevant ideal).
\end{example}

\begin{remark}
  Suppose that a finite length $A$-module $M$ has a pure free resolution of the form
\[
\cdots \to A(-d_{n+1})^{\beta_{n+1}} \to A(-d_n)^{\beta_n} \to \cdots
\to  A(-d_1)^{\beta_1} \to A(-d_0)^{\beta_0} \to M \to 0.
\]
We claim that for fixed $d_i$, the set of possible Betti numbers $(\beta_0,\beta_1,\dots)$ spans a 2-dimensional vector space and that it has a basis consisting of non-negative vectors corresponding to the cases $\beta_{n-1}=0$ and $\beta_{n-1}=1$.

First, the Hilbert series of $A$ is given by
\[
\hs_A(t) = \frac{1-t^2}{(1-t)^{n}} = \frac{1+t}{(1-t)^{n-1}}.
\]
Note that $\ker(A(-d_{n-2}))^{\beta_{n-2}} \to A(-d_{n-3})^{\beta_{n-3}})$ is a maximal Cohen--Macaulay module. In particular, its syzygy module is also maximal Cohen--Macaulay and has no free summands. From the theory of matrix factorizations, this implies that $\beta_{n-1+i} = \beta_{n-1}$ for all $i \ge 0$ and the resolution is
linear after the $(n-1)$st step, i.e., $d_{n-1+i} = d_{n-1} + i$ for $i \ge 0$. This yields:
\begin{align*}
\hs_M(t) &= \left( \sum_{i=0}^{n-2} (-1)^i \beta_i t^{d_i} \hs_A(t) \right) + (-1)^{n-1}
\beta_{n-1} t^{d_{n-1}} \hs_A(t) (1 - t + t^2 - t^3 + \cdots)\\
             &=  \left( \sum_{i=0}^{n-2} (-1)^i \beta_i \frac{t^{d_i} +  t^{d_i+1}}{(1-t)^{n-1}} \right)
               + (-1)^{n-1} \beta_{n-1} \frac{t^{d_{n-1}}}{(1-t)^{n-1}}. 
\end{align*}
Since $M$ is finite length, $\hs_M(t)$ is a polynomial, so $(1-t)^{n-1}$ divides the polynomial
\[
f(t) = \sum_{i=0}^{n-2} (-1)^i \beta_i(t^{d_i} + t^{d_i+1}) + (-1)^{n-1} \beta_{n-1} t^{d_{n-1}}.
\]
Equivalently, in terms of derivatives of $f$, we have $f(1) = f'(1) = \cdots = f^{(n-2)}(1) = 0$. This
gives us the system of linear equations
\begin{align*}
  \sum_{i=0}^{n-2} (-1)^i (2d_i - k + 2) \frac{d_i!}{(d_i - k + 1)!}
  \beta_i + (-1)^{n-1} \frac{d_{n-1}!}{(d_{n-1} - k)!} \beta_{n-1} &= 0 \quad (0
  \le k \le n-2).
\end{align*}

If $\beta_{n-1}=0$, then $f(t)$ is divisible by $1+t$, and the resulting problem reduces to the Herzog--K\"uhl equations for the polynomial ring in $n-1$ variables, in which case an explicit formula is given in \cite{herzogkuhl}.

Otherwise, we may assume $\beta_{n-1} \ne 0$, and up to rescaling, we may assume it is 1. Subtract the $\beta_{n-1}$ term from both sides of the above equations and then set $\beta_{n-1} = 1$. Then $\beta$ is the solution to $C\beta = (-1)^nx$ where
\begin{align*}
  C &= \begin{bmatrix} 2 & -2 & \cdots \\ 
  2d_0+1 & -(2d_1+1) & \cdots \\
  2d_0^2 & -2d_1^2 & \cdots \\
  (2d_0-1)d_0(d_0-1) & -(2d_1-1)d_1(d_1-1) &  \cdots \\
  \vdots \end{bmatrix},\qquad 
   x = \begin{bmatrix} 1 \\ d_{n-1} \\ d_{n-1}(d_{n-1}- 1) \\ d_{n-1}(d_{n-1}-1)(d_{n-1}-2) \\ \vdots \end{bmatrix}.
\end{align*}
Up to row operations, $C$ becomes $\pm 2 (d_{i-1}^j)_{1 \le i,j \le n-1}$, so in particular has nonzero determinant and $\beta$ is uniquely determined (up to scalar multiple). Since a non-negative solution exists by Theorem \ref{thm:quadricPieriResolutions}, this yields the desired claim. We have not attempted to get an explicit formula for $\beta$.
\end{remark}

\subsection{Rational normal curves}

For a composition $\alpha = (\alpha_1,\dots,\alpha_r)$, its ribbon diagram is the skew diagram with row lengths $\alpha_r,\dots,\alpha_1$ such that each row shares exactly one column with the next row.

Given another diagram $D$, the notation $D \odot \alpha$ will denote the skew partition obtained by attaching the top row of the ribbon diagram associated with $\alpha$ to the bottom row of the diagram $D$. Likewise, the notation $D \cdot \alpha$ is defined to be the skew partition obtained by concatenating the top row of the ribbon diagram associated with $\alpha$ to the bottom row of $D$.

\begin{example}
    Let $D = (3,3,2)/(1)$ and $\alpha = (2,2)$. Then
    \[
      \ytableausetup{boxsize=1em} 
      \alpha = \ydiagram{1+2,2}, \qquad D\cdot \alpha = \ydiagram{3+2,2+3,2+2,1+2,2}, \qquad D \odot \alpha = \ydiagram{4+2,3+3,1+4,2}. \qedhere
    \]
\end{example}
Let $D = \lambda / \mu$ be a skew-partition. By dualizing the Weyl module relations of \cite[Theorem II.3.16]{akin1982schur} the Schur module $\bbs_{D \odot (a)} (V)$ fits into the exact sequence
\begin{align*}
  0 \to \bbs_{D \odot (a)} \to S^{\lambda_1 - \mu_1} (V) \otimes \cdots \otimes S^{\lambda_n - \mu_n + a} (V) \to\\
 \bigoplus_{i=1}^{n-1} \bigoplus_{\ell=\lambda_{i}  - \lambda_{i+1} + 1}^{\lambda_i - \mu_i } S^{\lambda_1 - \mu_1} (V) \otimes \cdots \otimes S^{\lambda_i - \mu_i - \ell} \otimes S^{\lambda_{i+1} - \mu_{i+1} + \ell} (V) \otimes \cdots \otimes S^{\lambda_n - \mu_n + a} (V),
\end{align*}
where each component is induced by tensoring the appropriate identity maps with compositions of the form (let $\lambda'_n = \lambda_n+a$ and $\lambda'_i=\lambda_i$ otherwise)
\begin{align*}
  S^{\lambda_i - \mu_i} (V) \otimes S^{\lambda'_{i+1} - \mu_{i+1}} (V)
  &\xrightarrow{\Delta \otimes 1}  S^{\lambda_i - \mu_i - \ell} (V) \otimes S^\ell (V) \otimes S^{\lambda'_{i+1} - \mu_{i+1}} (V)\\
  &\xrightarrow{1 \otimes m}  S^{\lambda_i - \mu_i - \ell} (V) \otimes S^{\lambda'_{i+1} - \mu_{i+1} + \ell} (V),
\end{align*}
where $\Delta$ and $m$ denote the comultiplication and multiplication operations on symmetric powers, respectively. Notice that all of the above components commute with symmetric algebra multiplication on the $n$th tensor component, in which case we obtain a well-defined $S^\bullet (V)$-module structure on
\[
  \bbs_{D \odot \bullet} (V) := \bigoplus_{d \geq 0} \bbs_{D \odot (d)} (V).
\]
By convention, the module $\bbs_{D \odot \bullet} (V)$ has initial degree $0$ (with $(\bbs_{D \odot \bullet} (V))_0 = \bbs_D (V)$).

Assume now that $D \coloneq (\lambda_1,\lambda_2)$ is a partition with two rows. In characteristic $0$ almost all of the defining relations for the Schur module are redundant (it suffices to take the term with $\ell=\lambda_1-\lambda_2+1$) and we instead obtain an exact sequence of $S^\bullet (V)$-modules
\begin{align} \label{eqn:veronese-SES}
  0 \to \bbs_{\lambda \odot \bullet} (V) \to  S^\bullet (V)(\lambda_2) \otimes S^{\lambda_1} (V)\to S^\bullet (V) (\lambda_1+1) \otimes S^{\lambda_2 - 1} (V).
\end{align}
Note that the redundancy of the extra relations can be deduced either as a direct consequence of Zelevinsky's functor, or the description of the Jacobi--Trudi resolutions due to Akin \cite{akin1992complexes}. Furthermore, the rightmost map can be identified with the rightmost map of an EFW complex, so we conclude that the cokernel of \eqref{eqn:veronese-SES} is a finite length module.

\begin{prop}\label{prop:veroneseRes}
    Let $V$ be any vector space and $S^{(d)} \coloneq S^\bullet (V)^{(d)}$ denote the $d$th Veronese subalgebra. For positive integers $e_1, e_2$, set $D \coloneq (de_1 + de_2 - 1, de_2)/(d-1)$. Then there is an infinite $\GL (V)$-equivariant resolution
    $$ \cdots \to S^{(d)} (-e_1-e_2-  i) \otimes \bbs_{D \cdot (d^i)} (V) \to \cdots \to S^{(d)} (-e_1-e_2- 1) \otimes \bbs_{D \cdot (d)} (V)$$
    $$\to   S^{(d)} (-e_1-e_2) \otimes \bbs_{D} (V) \to S^{(d)} (-e_1) \otimes S^{d(e_1 + e_2 - 1) } \to S^{(d)} \otimes S^{d(e_2-1)} (V).$$
    Moreover, this is a resolution of a finite length $S^{(d)}$-module.
\end{prop}

\begin{proof}
  Consider \eqref{eqn:veronese-SES} with $\lambda = (de_1+de_2-d, de_2-d+1)$ and tensor with $S(-\lambda_1-1)$  (i.e., do a grading shift) to get
  \[
    0 \to \bbs_{\lambda \odot \bullet} (V)(-de_1-de_2+d-1) \to
    S^\bullet (V)(-de_1) \otimes S^{d(e_1+e_2-1)} (V) \to S^\bullet (V) \otimes S^{d(e_2-1)} (V).
  \]
  Now restrict to the terms whose degrees are a multiple of $d$ to get the exact sequence
  \[
    0 \to M(-e_1-e_2) \to S^{(d)} (V)(-e_1)\otimes S^{d(e_1+e_2-1)} (V) \to S^{(d)} (V) \otimes S^{d(e_2-1)} (V),
  \]
  where
  \[
    M = \bigoplus_{t \ge 0} \bbs_{\lambda \odot (d-1+td)}(V) = \bigoplus_{t \ge 0} \bbs_{D \odot (td)}(V).
  \]
  It follows from \cite[Theorem 6.28, Corollary 6.5]{vandebogert2023ribbon} that $M$ has a linear $S^{(d)}$-free resolution of the form
    $$\cdots \to S^{(d)} (- i) \otimes \bbs_{D \cdot (d^i)} (V) \to \cdots \to S^{(d)} (- 1) \otimes \bbs_{D \cdot (d)} (V) \to S^{(d)} \otimes \bbs_D (V).$$
    Splicing the above two complexes together and twisting the grading yields the desired resolution.
    The claim about the cokernel having finite length follows from the discussion immediately following \eqref{eqn:veronese-SES}.
\end{proof}

Recall that by Proposition \ref{prop:veronese}, for any partition $\lambda = (\lambda_1 , \lambda_2)$ there is an equality
$$\bbs_{\lambda}^{S^{(d)}} = \bbs_{(d \lambda_1 + d-1 , d \lambda_2) / (d-1)} (V).$$

\begin{remark}
    When $\dim V = 2$, the resolutions of Proposition~\ref{prop:veroneseRes} give another explanation of the correspondence discovered in \cite{BS-RNC} between degree sequences $(d_0 , d_1, d_2 ; d-1)$ over the rational normal curve and degree sequences $(d \cdot d_0 , d \cdot d_1, d \cdot d_2 - (d-1))$: first, note that since $\dim V = 2$ there is an isomorphism
    $$\bbs_{D \cdot (d^i)} \cong \bbs_{(d^{i-1} , d-1)}^{\rm rib} (V) \otimes \bbs_{(de_1 + de_2 , de_2 + 1)/(d)} (V),$$
    where $\bbs_{(d^{i-1} , d-1)}^{\rm rib} (V)$ denotes the ribbon Schur functor associated to the composition $(d^{i-1} , d-1)$. Thus the Betti table corresponding to the tail of the resolution of Proposition~\ref{prop:veroneseRes} is, up to a scalar multiple, the same as that for the resolution of the maximal Cohen--Macaulay module $S^{(d-1 \mod d)} \cong \bigoplus_{k \geq 0} S^{dk+d-1} (V)$. Similarly, there is an isomorphism
    $$\bbs_D (V) \cong S^{d-1} (V) \otimes \bbs_{(de_1 + de_2 -d , de_2 - (d-1))} (V),$$
    implying that the second Betti number is a scalar multiple of the second Betti number corresponding to the degree sequence $(0, de_1 , de_1 + de_2)$ over the polynomial ring $S^\bullet (V)$. Finally, the $0$ and $1$ Betti numbers over $S^{(d)}$ are directly equal to those occurring over the polynomial ring $S^\bullet (V)$.
\end{remark}

\begin{theorem}
    Let $A \coloneq S^{(d)}$ denote the coordinate ring of the degree $d$ rational normal curve and set $G = \GL_2$. Pick positive integers $e_1 , e_2, e_3$, set $e_i=1$ for $i \ge 4$, set $d_i = \sum_{j=1}^i e_j$  and define $\lambda^{(i)}$ as in \S\ref{sec:EFW}.
    \begin{enumerate}
        \item If $e_3>1$, there is a finite length $G$-equivariant pure free resolution 
          \[
            0 \to A (-d_{2}) \otimes \bbs_{\lambda^{(2)}}^A \to  A (-d_{1}) \otimes \bbs_{\lambda^{(1)}}^A \to A \otimes \bbs_{\lambda^{(0)}}^A
          \]
          with degree sequence $0,d_1 , d_2 $. 

      \item If $e_3=1$, there is an infinite length $G$-equivariant pure free resolution
          \begin{align*}
\cdots \to            A (-d_{n}) \otimes \bbs_{\lambda^{(n)}}^A
 \to \cdots \to A (-d_{1}) \otimes \bbs_{\lambda^{(1)}}^A \to A \otimes \bbs_{\lambda^{(0)}}^A
          \end{align*}
          with infinite degree sequence $0,d_1 , d_2 ,  d_{2}+1 ,d_{2}+2 , \dots$.
        \end{enumerate}
              In both cases, the module being resolved has finite length.
\end{theorem}

\begin{proof}
  First, we spell out the partitions $\lambda^{(i)}$ for $i=0,1,2,3$:
  \begin{align*}
    \lambda^{(0)} &= (e_2+e_3-2, e_3-1)\\
    \lambda^{(1)} &= (e_1+e_2+e_3-2, e_3-1)\\
    \lambda^{(2)} &= (e_1+e_2+e_3-2, e_2+e_3-1)\\
    \lambda^{(3)} &= (e_1+e_2+e_3-2, e_2+e_3-1,e_3).
  \end{align*}

  \textbf{Proof of (1):} First consider the EFW complex over $S = S^\bullet(V)$ from \S\ref{sec:EFW} with degree sequence $(0,dd_1,dd_2)$. This is a resolution of the form
  \[
     0 \to S (-dd_{2}) \otimes \bbs_{\mu^{(2)}} \to  S (-dd_{1}) \otimes \bbs_{\mu^{(1)}} \to S \otimes \bbs_{\mu^{(0)}},
   \]
where
  \begin{align*}
    \mu^{(0)} &= (de_2+de_3-2, de_3-1)\\
    \mu^{(1)} &= (de_1+de_2+de_3-2, de_3-1)\\
    \mu^{(2)} &= (de_1+de_2+de_3-2, de_2+de_3-1).
  \end{align*}
  Then $\bbs^A_{\lambda^{(3)}}=0$, and for $i=0,1,2$ we have isomorphisms
  \[
    \bbs^A_{\lambda^{(i)}}(V) \cong \bbs_{\mu^{(i)}}(V) \otimes \bbs_{(2d-1,2d-1)}(V)^* \otimes S^{d-1}(V).
  \]
  Thus the desired complex is obtained from the EFW complex, restricting to terms whose degrees are divisible by $d$, and tensoring with $\bbs_{(2d-1,2d-1)}(V)^* \otimes S^{d-1}(V)$.

  \textbf{Proof of (2):} Use Proposition~\ref{prop:veroneseRes} with $\dim V =2$.
\end{proof}

In general, we do not know if an identical construction to the above yields pure free resolutions over arbitrary Veronese subalgebras. One issue with the general case is that ``Pieri'' maps are not uniquely defined for skew Schur functors, and we moreover cannot employ Schur's lemma to deduce that the resulting sequences of maps yield a complex. 

\bibliography{biblio}

\end{document}